\documentclass[reqno,dvips]{amsart}

\usepackage{mathrsfs}
\usepackage{amsmath}
\usepackage{amssymb}
\usepackage{cite}
\usepackage{latexsym}
\usepackage[pdftex]{graphicx}

\usepackage{color}
\usepackage[normalem]{ulem}

\usepackage{comment}

\theoremstyle{plain}
\newtheorem{thm}{Theorem}[section]
\newtheorem{lem}[thm]{Lemma}
\newtheorem{cor}[thm]{Corollary}
\newtheorem{dfn}[thm]{Definition}
\newtheorem{prop}[thm]{Proposition}
\newtheorem{rmk}[thm]{Remark}

\newcommand{\blue}[1]{\textcolor{blue}{#1}}

\newcommand{\green}[1]{\textcolor{green}{#1}}

\def\D{\mathrm{D}}

\def\P{\mathrm{P}}
\def\R{\mathrm{R}}

\def\T{\mathrm{T}}
\def\d{\mathrm{d}}

\def\Nset{\mathbb{N}}

\def\Qset{\mathbb{Q}}
\def\Rset{\mathbb{R}}
\def\Sset{\mathbb{S}}
\def\Tset{\mathbb{T}}
\def\Zset{\mathbb{Z}}

\def\M{\mathscr{M}}

\def\I{\mathscr{I}}

\def\P{\mathscr{P}}

\def\h{\mathrm{h}}

\DeclareMathOperator{\sech}{sech}
\DeclareMathOperator{\csch}{csch}
\DeclareMathOperator{\sn}{sn}
\DeclareMathOperator{\cn}{cn}
\DeclareMathOperator{\dn}{dn}

\DeclareMathOperator{\rank}{rank}

\def\epsilon{\varepsilon}

\makeatletter
\@addtoreset{equation}{section}
\@namedef{subjclassname@2020}{\textup{2020} Mathematics Subject Classification}
\makeatother

\def\theequation{\arabic{section}.\arabic{equation}}

\begin{document}


\title[Obstructions to Integrability of Nearly Integrable Systems]%
{Obstructions to Integrability of Nearly Integrable Dynamical Systems near
 Regular Level Sets}
\thanks{}

\author{Shoya Motonaga}
\author{Kazuyuki Yagasaki}

\address{Department of Applied Mathematics and Physics, Graduate School of Informatics,
Kyoto University, Yoshida-Honmachi, Sakyo-ku, Kyoto 606-8501, JAPAN}

\email{mnaga@amp.i.kyoto-u.ac.jp (S.~Motonaga)}

\email{yagasaki@amp.i.kyoto-u.ac.jp (K.~Yagasaki)}

\date{\today}
 
\begin{abstract}
We study the existence of real-analytic first integrals and real-analytic integrability for perturbations
 of integrable systems in the sense of Bogoyavlenskij
 including non-Hamiltonian ones.
We especially assume that there exists a family of periodic orbits
 on a regular level set of the first integrals having a connected and compact component
 and give sufficient conditions for nonexistence of the same number of real-analytic first integrals
 in the perturbed systems as the unperturbed ones
 and for their real-analytic nonintegrability near the level set
 such that the first integrals and commutative vector fields
 depend analytically on the small parameter.
We compare our results with classical results of Poincar\'e and Kozlov
 for systems written in action and angle coordinates
 and discuss their relationships with the subharmonic and homoclinic Melnikov methods
 for periodic perturbations of single-degree-of-freeedom Hamiltonian systems.
We illustrate our theory for three examples
 containing the periodically forced Duffing oscillator.
\end{abstract}
\maketitle


\section{Introduction}

In his famous memoir \cite{P90},
 which was related to a prize competition celebrating the 60th birthday of King Oscar II,
 Henri Poincar\'e studied two-degree-of-freedom Hamiltonian systems
 depending on a small parameter,
 say $\epsilon$ here although he used the letter $\mu$ instead,
 such that they are integrable when $\epsilon=0$,
 and showed the nonexistsnce of first integrals
 which are analytic in the state variables and parameter $\epsilon$
 and functionally independent of Hamiltonians,
 under some nondegenerate conditions.
If there exists such a first integral,
 then the Hamiltonian systems are integrable for $|\epsilon|\ge 0$ sufficiently small 
 in the sense of Liouville \cite{A89,M99}.
The result was improved significantly
 in the first volume of his masterpieces \cite{P92} published two years later,
 so that more-degree-of-freedom Hamiltonian systems can be treated.
Using these results,
 he discussed the nonexistence of such first integrals
 in the restricted planar and spacial three-body problem there.
See also \cite{B96} for an account of his work
 from a mathematical and historical perspectives.
Subsequently, his results were sophisticated and generalized
 to non-Hamiltonian systems \cite{K83,K96}.
In particular, Kozlov \cite{K96} treated multi-dimensional systems of the form
\begin{equation}
\dot{I}=\epsilon h(I,\theta;\epsilon),\quad
\dot{\theta}=\omega(I)+\epsilon g(I,\theta;\epsilon),\quad
(I,\theta)\in\Rset^\ell\times\Tset^m,
\label{eqn:aasys}
\end{equation}
where $\epsilon$ is a small parameter such that $|\epsilon|\ll 1$,
 $\Tset^m=\prod_{j=1}^m\Sset^1$ with $\Sset^1=\Rset/2\pi\Zset$ is an $m$-dimensional torus
 and $h(I,\theta;\epsilon)$, $\omega(I)$ and $g(I,\theta;\epsilon)$ are analytic in $(I,\theta,\epsilon)$.
Note that the system \eqref{eqn:aasys} is Hamiltonian if $\ell=m$ as well as $\epsilon=0$ or
\[
\D_I h(I,\theta;\epsilon)=-\D_\theta g(I,\theta;\epsilon),
\]
and non-Hamiltonian if not.
When $\epsilon=0$, Eq.~\eqref{eqn:aasys} becomes
\begin{equation}
\dot{I}=0,\quad
\dot{\theta}=\omega(I)
\label{eqn:aasys0}
\end{equation}
which we refer to as the \emph{unperturbed system} for \eqref{eqn:aasys}.
We often use this terminology for other systems below.
Here we state some details of his result.

We expand $h(I,\theta;0)$ in Fourier series as
\begin{equation}
h(I,\theta;0)=\sum_{r\in\Zset^m}\hat{h}_r(I)\exp(i r\cdot\theta),
\label{eqn:hath}
\end{equation}
where $\hat{h}_r(I)$, $r\in\Zset^m$, are the Fourier coefficients
 and ``$\cdot$'' represents the inner product.
We assume the following for \eqref{eqn:aasys}:
\begin{enumerate}
\setlength{\leftskip}{-1em}
\item[\bf(K1)]
The system~\eqref{eqn:aasys} has $s$ first integrals $F_j(I,\theta; \epsilon)$, $j=1,\ldots,s$, 
 which are analytic in $(I,\theta,\epsilon)$;
\item[\bf(K2)]
If $r\in\Zset^m$ and $r\cdot\omega(I)=0$ for any $I\in\Rset^\ell$, then $r=0$.
\end{enumerate}
If assumption~(K2) holds,
 then we say that the unperturbed system \eqref{eqn:aasys0} is \emph{nondegenerate}.
Under 
(K1) and (K2)
 we can show that $F_j(I,\theta;0)$, $j=1,\ldots,s$, are independent of $\theta$
  (see Lemma~1 in Section~1 of Chapter IV of \cite{K96}),
 and write $F_{j0}(I)=F_j(I,\theta;0)$ and $F_0(I)=(F_{10}(I),\ldots,F_{s0}(I))$.
We refer to $\P_s\subset\Rset^\ell$ as a \emph{Poincar\'e set}
 if for each $I\in\P_s$ there exists linearly independent vectors $r_j\in\Zset^m$, $j=1,\ldots,\ell-s$,
 such that
\begin{enumerate}
\setlength{\leftskip}{-1.8em}
\item[(i)]
$r_j\cdot\omega(I)=0$, $j=1,\ldots ,\ell-s$;
\item[(ii)]
$\hat{h}_{r_j}(I)$, $j=1,\ldots ,\ell-s$, are linearly independent.
\end{enumerate}
Let $U$ be a domain in $\Rset^\ell$.
A set $\Delta\subset U$ is called a \emph{key set} (or \emph{uniqueness set}) for $C^\omega(U)$
 if any analytic function vanishing on $\Delta$ vanishes on $U$.
For example, any dense set in $U$ is a key set for $C^\omega(U)$.
In this situation, we have the following theorem
 (see Section~1 of Chapter IV of \cite{K96} for its proof).

\begin{thm}[Kozlov]
\label{thm:Kozlov}
Suppose that assumptions~(K1) and (K2) hold,
 the Jacobian matrix $\D F_0(I)$ has the maximum rank at a point $I_0\in\Rset^\ell$
 and a Poincar\'e set $\P_s\subset U$ is a key set for $C^\omega(U)$,
 where $U$ is a neighborhood of $I_0$.
Then the system \eqref{eqn:aasys} has no first integral
 which is real-analytic in $(I,\theta,\epsilon)$
 and functionally independent of $F_j(I,\theta;\epsilon)$, $j=1,\ldots,s$, near $\epsilon=0$.
\end{thm}

A version of Theorem~\ref{thm:Kozlov} for the Hamiltonian case $\ell=m$
 was given in \cite{K83} (see also Theorem~7.1 of \cite{AKN06}).
The Hamiltonian case of $s=1$ with a dense Poincar\'e set in Theorem~\ref{thm:Kozlov}
 was treated by Poincar\'e 
 for $\ell=m\ge 2$ in \cite{P92}.
When $s=0$ in (K1),
 Theorem~\ref{thm:Kozlov} means that under its hypotheses
 there exists no first integral which is analytic in $(I,\theta,\epsilon)$.
When $s=1$ in (K1), which always occurs if the system \eqref{eqn:aasys} is Hamiltonian,
 it means that under its hypotheses
 there exists no first integral which is analytic
 in $(I,\theta)\in U\times\Tset^m$ and $\epsilon$ near $\epsilon=0$
 and functionally independent of $F_1(I,\theta,\epsilon)$.
When $s=m-1$ in (K1),
 if besides (K1) and (K2) there exists a key set $\Delta$ for $C^\omega(U)$
 such that $r\cdot\omega(I)=0$ and $h_r(I)\neq 0$ on $\Delta$ for some $r\in\Zset^m\setminus \{0\}$,
 there exists no additional first integral which is analytic in $(I,\theta,\epsilon)$.
For Hamiltonian systems, if they are not Liouville-integrable
 for any $|\epsilon|>0$ sufficiently small,
 then there does not exist such a first integral.
For non-Hamiltonian systems
 this may be true in an appropriate meaning 
 but some additional ingredients are needed as seen below.
 
In this paper we study more general dynamical systems of the form
\begin{equation}
\dot{x} = X_\epsilon(x),\quad x\in \mathscr{M},
\label{eqn:p-gsys}
\end{equation}
where $\epsilon$ is a small parameter such that $|\epsilon|\ll 1$,
 $\mathscr{M}$ is an $n$-dimensional analytic manifold for $n\ge 2$
 and  the vector field $X_\epsilon$ is analytic in $x$ and $\epsilon$.
Let $X_\epsilon=X^0+\epsilon X^1+O(\epsilon^2)$ for $|\epsilon|>0$ sufficiently small. 
When $\epsilon=0$, the system \eqref{eqn:p-gsys} becomes
\begin{equation}
\dot{x} = X^0(x),
\label{eqn:gsys0}
\end{equation}
which is assumed to be \emph{analytically $(q,n-q)$-integrable}
 in the following meaning of Bogoyavlenskij \cite{B98} for some positive integer $q\le n$.

\begin{dfn}[Bogoyavlenskij]
\label{dfn:1}
The system \eqref{eqn:gsys0} is called \emph{$(q,n-q)$-integrable} or simply \emph{integrable} 
 if there exist $q$ vector fields $Y_1(x)(:= X^0(x)),Y_2(x),\dots,Y_{q}(x)$
 and $n-q$ scalar-valued functions $F_1(x),\dots,F_{n-q}(x)$
 such that the following two conditions hold:
\begin{enumerate}
\setlength{\leftskip}{-1.8em}
\item[\rm(i)]
$Y_1(x),\dots,Y_q(x)$ are linearly independent almost everywhere and commute with each other,
 i.e., $[Y_j,Y_k](x)\equiv 0$ for $j,k=1,\ldots,q$, where $[\cdot,\cdot]$ denotes the Lie bracket;
\item[\rm(ii)]
$F_1(x),\dots, F_{n-q}(x)$ are functionally independent, i.e.,
 $dF_1(x), \dots, dF_{n-q}(x)$ are linearly independent almost everywhere,
 and $F_1(x),\dots,F_{n-q}(x)$ are first integrals of $Y_1, \dots,Y_q$,
 i.e., $dF_k(Y_j)=0$ for $j=1,\ldots,q$ and $k=1,\ldots,n-q$.
\end{enumerate}
If $Y_1,Y_2,\dots,Y_q$ and $F_1,\dots,F_{n-q}$ are analytic (resp. meromorphic), 
 then Eq.~\eqref{eqn:gsys0} is said to be \emph{analytically} (resp. \emph{meromorphic})
 \emph{integrable}.
\end{dfn}

Definition~\ref{dfn:1} is considered as a generalization of 
 Liouville-integrability for Hamiltonian systems \cite{A89,M99}
 since an $m$-degree-of-freedom Liouville-integrable Hamiltonian system with $m\ge 1$
 has not only $m$ functionally independent first integrals
 but also $m$ linearly independent commutative (Hamiltonian) vector fields
 generated by the first integrals.
The ($\ell+m$)-dimensional system~ \eqref{eqn:aasys0} is $(m,\ell)$-integrable
 in the Bogoyavlenskij sense.

The system \eqref{eqn:p-gsys} is regarded
 as a perturbation of the analytically integrable system \eqref{eqn:gsys0}.
If there exist $q'$ linearly independent, analytic commutative vector fields
 and $n-q'$ functionally independent, analytic first integrals
 depending on $\epsilon$ analytically near $\epsilon=0$ for some positive integter $q'\le n$,
 then the perturbed system \eqref{eqn:p-gsys} is analytically $(q',n-q')$-integrable
 for $|\epsilon|\ge 0$ sufficiently small.
We especially assume that there exists a $q$-parameter family of periodic orbits
 on a regular level set of the $n-q$ first integrals having a connected and compact component
 (see Section~2 for our precise definitions containing this one)
 and give sufficient conditions for nonexistence of such $n-q$ real-analytic first integrals
 and for real-analytic nonintegrability of the perturbed system \eqref{eqn:p-gsys} near the regular level set
 such that the first integrals and commutative vector fields
 depend analytically on $\epsilon$ near $\epsilon=0$.
The persistence of such first integrals and commutative vector fields 
 in the perturbed system \eqref{eqn:p-gsys} along with that of periodic and homoclinic orbits 
 was previously discussed in \cite{MY}.
Our approach is different from those of Poincar\'e \cite{P90,P92} and Kozlov \cite{K83,K96}
 and based on the technique of \cite{MY}.
Recently, another sufficient condition for nonintegrability
 of nearly integrable dynamical systems near resonant periodic orbits 
 was also obtained using a different approach in \cite{Y21a}
 and the theory was applied to prove the nonintegrability of the restricted three-body problem 
 in \cite{Y21b}, when the independent and state variables are extended to complex ones.

The unperturbed system~\eqref{eqn:gsys0} can be transformed to \eqref{eqn:aasys0}
 under our assumptions, as stated in Proposition~\ref{LAJ} below.
However, Theorem~\ref{thm:Kozlov} only says in our setting for \eqref{eqn:aasys}
 that there exists no such additional first integral 
 even if its hypotheses hold. 
In particular, it does not allow us to determine the nonintegrability of \eqref{eqn:p-gsys}
 in the Bogoyavlenskij sense when it is non-Hamiltonian.
We also describe a consequence of our results to \eqref{eqn:aasys}
 and show how it improves the results of Poincar\'e \cite{
 P92} and Kozlov \cite{K83,K96}.

Moreover, we discuss a relationship between our results
 and the subharmonic and homoclinic Melnikov methods \cite{GH83,W90,Y96}
 for time-periodic perturbations of single-degree-of-freedom analytic Hamiltonian systems,
 which can be transformed to 
 the form of \eqref{eqn:p-gsys} with $\ell=1$ and $m=2$,
 i.e, $(2,1)$-integrable, and have a one-parameter family of periodic orbits when $\epsilon=0$.
As well known, if the subharmonic Melnikov functions have a simple zero,
 then the corresponding unperturbed periodic orbits persist in the perturbed systems. 
Morales-Ruiz \cite{M02} studied the Hamiltonian perturbation case
 in which the unperturbed systems have homoclinic orbits,
 and showed a relationship between their nonintegrability
 and a version due to Ziglin \cite{Z82b} of the homoclinic Melnikov method 
 when the independent and state variables are extended to complex ones
 and the small parameter $\epsilon$ is also regarded as a state variable.
More concretely, under some restrictive conditions,
 based on the generalized version due to Ayoul and Zung \cite{AZ10}
 of the Morales-Ramis theory \cite{M99,MR01},
 which provides a sufficient condition for nonintegrability of autonomous dynamical systems,
 he essentially proved that they are meromorphically nonintegrable
 if the homoclinc Melnikov functions
 which are obtained as integrals along closed loops on the complex plane
 are not identically zero.
Here the version of the Melnikov method enables
 us to detect transversal self-intersection of complex separatrices of periodic orbits
 and to prove its analytic nonintegrability, unlike the standerd version \cite{GH83,M63,W90}.
Here we prove two variants of Morales-Ruiz \cite{M02} for periodic orbits:
 if the subharmonic Melnikov functions for a dense set of the unperturbed periodic orbits
 are not identically zero,
 then there exists no first integral depending analytically on $\epsilon$ near $\epsilon=0$;
 and if the `standard' homoclinic Melnikov functions \cite{GH83,M63,W90} are not identically zero,
 then the perturbed systems are not Bogoyavlenskij-integarble
 such that the commutative vector fields and first integrals
 depend analytically on $\epsilon$ near $\epsilon=0$.

We illustrate our theory for three examples:
Simple pendulum with a constant torque, 
 second-order coupled oscillators and the periodically forced Duffing oscillator \cite{GH83,H79,W90}.
Real-analytic nonintegrability is discussed in 
special cases of the second and third examples
 while existence of real-analytic first integrals in the rest.
In particular, the special case of the third one is shown to be nonintegrable in the above meaning
 even if it does not have transverse homoclinic orbits to a periodic orbit.
 
This paper is organized as follows:
In Section~\ref{Main results}, we state our precise assumptions and main theorems.
In Section~\ref{pf}, we give proofs of the main theorems.
We describe a consequence of our results to \eqref{eqn:aasys} in Section~4
 and discuss a relationship between our results and the subharmonic and homoclinic Melnikov methods
 for time-periodic perturbations of single-degree-of-freedom Hamiltonian systems in Section~5.
Finally, we give the three examples in Section~6.


\section{Main Results}\label{Main results}
In this section, we state our main results for \eqref{eqn:p-gsys}.
We first make the following assumptions on the unperturbed system \eqref{eqn:gsys0}:
\begin{enumerate}
\setlength{\leftskip}{-1em}
\item[\bf(A1)]
For some positive integer $q<n$,
 the system \eqref{eqn:gsys0} is analytically $(q,n-q)$ Bogoyavlenskij-integrable,
 i.e., there exist $q$ analytic vector fields $Y_1(x)(:=X^0(x)),\ldots,Y_q(x)$
 and $n-q$ analytic scalar-valued functions $F_1(x),\ldots,F_{n-q}(x)$
 such that  conditions~(i) and (ii) of Definition~\ref{dfn:1} hold.
\item[\bf(A2)]
Let $F(x)=(F_1(x),\ldots,F_{n-q}(x))$
 .
There exists a regular value $c\in\Rset^{n-q}$ of $F$,
 i.e., $\rank dF(x)=n-q$ when $F(x)=c$,
 such that $F^{-1}(c)$ has a connected and compact component
 and $Y_1(x), \ldots, Y_q(x)$ are linearly independent on $F^{-1}(c)$.
\end{enumerate}
Henceforth we assume without loss of generality that $F^{-1}(c)$ is connected and compact itself
 in (A2), by reducing the domain of $F$ if necessary.
Under 
 (A1) and (A2) we have the following result
 like a well-known 
 theorem for Hamiltonian systems \cite{A89} 
 (see \cite{B98, Z06,Z18} for the details).

\begin{prop}[Liouville-Miuner-Arnold-Jost]\label{LAJ}
Suppose that assumptions~(A1) and (A2) hold.
Then we have the following:
\begin{enumerate}
\setlength{\leftskip}{-2em}
\item[\rm(i)]
The level set $F^{-1}(c)$ is analytically diffeomorphic to the $q$-dimensional torus
 $\Tset^q$;
\item[(ii)]
There exists an analytic diffeomorphism $\varphi:U\times\Tset^q 
 \to\mathcal{U}$, where $U$ and $\mathcal{U}$ are, respectively,
 neighborhoods of $I=I_0$ in $\Rset^{n-q}$ and of $F^{-1}(c)$ in $\M$ for some $I_0\in\Rset^{n-q}$,
 such that 
\begin{enumerate}
\setlength{\leftskip}{-1.8em}
\item[(iia)]
$\varphi(\{I_0\}\times\Tset)=F^{-1}(c)$;
\item[(iib)]
$F\circ\varphi(I, \theta)$ depends only on $I$,
 where $(I, \theta)\in U\times\Tset^q$;
\item[(iic)]
The flow of $X^0$ on $\mathcal{U}$
 is analytically conjugate to that of \eqref{eqn:aasys0} with $\ell=n-q$ and $m=q$
 on $U\times \Tset^q$.
\end{enumerate}
\end{enumerate}
\end{prop}

The variables $I$ and $\theta$ are called the \emph{action} and \emph{angle variables}
 as in Hamiltonian systems,
 and $\omega(I)$ is referred to as the \emph{angular frequency vector}.
Let $\omega_j(I)$ be the $j$th component of $\omega(I)$ for $j=1,\ldots,q$.
Henceforth $U$ denotes the neighborhood of $I=I_0$ in Proposition~\ref{LAJ}.
Moreover, we assume the following on \eqref{eqn:aasys0} along with (K2):

\begin {enumerate}
\setlength{\leftskip}{-1em}
\item[\bf(A3)]
There exists a key set $D_\R$ for $C^\omega(U)$
 such that for $I\in D_\R$ a resonance of multiplicity $q-1$,
\[
\dim_\Qset \langle \omega_1(I), \ldots, \omega_q(I) \rangle =1,
\]
occurs with $\omega(I)\neq 0$,
 i.e., there exists a positive constant $\omega_0(I)$ depending on $I$ such that
\[
\frac{\omega(I)}{\omega_0(I)}\in\Zset^q\setminus\{0\}.
\]
We choose $\omega_0(I)$ as large as possible below.
\end{enumerate}

We easily see that if $\rank\D\omega(I^*)=q$ for some $I^*\in U$, 
 then both (K2) and (A3) hold.
Assumption~(A3) also implies that if $I\in D_\R$,
 then the system \eqref{eqn:aasys0} has a $q$-parameter family of periodic orbits
\begin{equation}
(I, \theta)=(I, \omega(I)t+\tau),\quad
\tau\in\Tset^q,
\label{eqn:po}
\end{equation}
with the period $T^I=2\pi/\omega_0(I)$.
Note that the periodic orbits \eqref{eqn:po} is parameterized by $(q-1)$ parameters essentially 
 since two periodic orbits $(I,\omega(I)t+\tau)$ and $(I,\omega(I)t+\tau_0)$
 represent the same orbit if $\tau-\tau_0=\omega(I)t_0$ for some $t_0\in\Rset$.
We also have a $q$-parameter (but essentially $(q-1)$-parameter) family of periodic orbits
\[
\gamma_\tau^I(t)=\varphi(I, \omega(I)t+\tau),\quad
 (I,\tau)\in D_\R\times\Tset^q,
\]
with the period $T^I$ in the unperturbed system \eqref{eqn:gsys0}.
Define the integrals
\begin{align}\label{eqn:i}
\I_{F_k}^I(\tau):=\int_0^{T^I} dF_j(X^1)_{\gamma_\tau^I (t)} dt,\quad
 k=1,\ldots,n-q,
\end{align}
for $I\in D_\R$ and set $\I_F^I(\tau):=(\I_{F_1}^I(\tau),\ldots,\I_{F_{n-q}}^I(\tau))$.
Note that
\[
\I_{F_k}^I(\tau+\omega(I)t)=\I_{F_k}^I(\tau)
\]
for $\tau\in\Tset^q$ and $t\in\Rset$.
We state the first of our main results as follow.

\begin{thm}\label{thm:main-fi}
Suppose that assumptions (A1)-(A3) and (K2) hold.
If there exists a key set $D\subset D_\R$ for $C^\omega(U)$
 such that $\I_F^I(\tau)$ is not identically zero for any $I\in D$,
 then the perturbed system \eqref{eqn:p-gsys} does not have $n-q$ real-analytic first integrals
 in a neighborhood of $F^{-1}(c)$ near $\epsilon=0$
 such that they are functionally independent for $|\epsilon|\neq 0$
 and depend analytically on $\epsilon$.
\end{thm}
We prove Theorem~\ref{thm:main-fi} in Section~\ref{pf:thm1}.

We next consider a special case in which
 Eq.~\eqref{eqn:p-gsys} is a two- or more-degree-of-freedom Hamiltonian system.
For an integer $m\ge 2$,
 let $(\M,\Omega)$ be a $2m$-dimensional analytic symplectic manifold
 with a symplectic form $\Omega$,
 and let $H_\epsilon(x)=H^0(x)+\epsilon H^1(x)+O(\epsilon^2)$
 be an analytic Hamiltonian for \eqref{eqn:p-gsys}
 and depend analytically on $\epsilon$ near $\epsilon=0$.
We assume the following:
\begin {enumerate}
\setlength{\leftskip}{-0.6em}
\item[\bf
(A1')]
The unperturbed Hamiltonian system \eqref{eqn:gsys0} with the Hamiltonian $H^0(x)$
 is real-analytically Liouville-integrable, i.e.,
 there exist $m$ functionally independent analytic first integrals $F_1(x)(:=H^0(x )),F_2(x),\ldots,F_m(x)$
 such that they are pairwise Poisson commutative,
 i.e., $\{F_j,F_k\}=0$ for $j,k=1,\ldots,m$,
 where $\{\cdot,\cdot\}$ denotes the Poisson bracket for the symplectic form $\Omega$.
\end{enumerate}
Assumption~(A1') means (A1) with $q=m$ and $Y_j=X_{F_j}$ for $j=1, \ldots, m$,
 where $X_{F_j}$ denotes the Hamiltonian vector field for the Hamiltonian $F_j(x)$.
So we immediately obtain the following from Theorem~\ref{thm:main-fi}.

\begin{cor}\label{cor:Hsys}
Suppose that (A1'), (A2), (A3) and (K2) hold.
If there exists a key set $D\subset D_\R$ for $C^\omega(U)$
 such that $\I_F^I(\tau)$ with $X_1=X_{H^1}$ is not identically zero for any $I\in D$,
 then the perturbed Hamiltonian system for the Hamiltonian $H_\epsilon(x)$
 is not real-analytically Liouville-integrable in a neighborhood of $F^{-1}(c)$ near $\epsilon=0$
 such that $m-1$ additional first integrals also depend analytically on $\epsilon$.
\end{cor}

For integrability of non-Hamiltonian systems \eqref{eqn:p-gsys},
 we have to consider not only first integrals but also commutative vector fields.
So we need the following assumption additionally:
\begin {enumerate}
\setlength{\leftskip}{-1em}
\item[\bf(A4)]
For some $I^*\in U$, the Jacobian matrix $\D\omega(I^\ast)$ is injective, i.e., 
\[
 \rank\D\omega(I^\ast)=n-q.
\]
\end{enumerate}
Note that (A4) holds only when $n-q\le q$.
Finally, we state the second main result as follows.

\begin{thm}\label{thm:main-integ}
Suppose that assumptions (A1)-(A4) and (K2) hold.
If there exists a key set $D\subset D_{\R}$ for $C^\omega(U)$
 such that $\I_F^I(\tau)$ is not constant for any $I\in D$,
 then for $|\epsilon|\neq 0$ sufficiently small
 the perturbed system \eqref{eqn:p-gsys} is not real-analytically integrable 
 in the Bogoyavlenskij sense near $F^{-1}(c)$
 such that the first integrals and commutative vector fields
 depend analytically on $\epsilon$ near $\epsilon=0$:
 In particular there exists only such $n-q-1$ first integrals
 and such $q-1$ commutative vector fields at most.
\end{thm}

\begin{rmk}\
\label{rmk:2a}
\begin{enumerate}
\setlength{\leftskip}{-2em}
\item[\rm(i)]
Theorems~\ref{thm:main-fi} and \ref{thm:main-integ} exclude even the possibility that
 the first integrals are functionally independent for $|\epsilon|\neq 0$ sufficiently small
 but not for $\epsilon=0$.
\item[\rm(ii)]
Theorem~\ref{thm:main-integ} does not exclude the possibility that
 the system \eqref{eqn:p-gsys} is real-analytically integrable
 in the Bogoyavlenskij sense near $F^{-1}(c)$
 such that in a punctured neighborhood of $\epsilon=0$
 there exist such $n-q'(\ge 0)$ first integrals and $q'$ commutative vector fields for some $q'>q$.
\end{enumerate}
\end{rmk}

We prove Theorem~\ref{thm:main-integ} in Section~\ref{pf:thm2}.


\section{Proofs of the main theorems}\label{pf}
In this section,
 we prove Theorems~\ref{thm:main-fi} and \ref{thm:main-integ}.
The proofs are based on the results of \cite{MY}.
Henceforth we mean real-analyticity when saying analyticity.

\subsection{Proof of Theorem~\ref{thm:main-fi}}\label{pf:thm1}
We first prove Theorem~\ref{thm:main-fi}. 
Henceforth we assume that conditions~(A1)-(A3) and (K2) hold.
We begin with the following lemma.

\begin{lem}\label{erg}
Suppose that $G_1(x),\ldots,G_k(x)$ are analytic first integrals of the unperturbed system \eqref{eqn:gsys0}
 near $F^{-1}(c)$ for $k\ge 1$
 and they may be functionally dependent.
Let $G(x)=(G_1(x),\ldots,G_k(x))$
 and let $\phi_t$ denote the flow generated by \eqref{eqn:gsys0}. 
Then the following statements hold:
\begin{enumerate}
\setlength{\leftskip}{-1.8em}
\item[(i)]
There exists an analytic map $\psi:F(\mathcal{U})\to\Rset^k$ satisfying $G=\psi\circ F$
 in a neighborhood $\mathcal{U}$ of $F^{-1}(c)$ in $\M$.
\item[(ii)]
If $dG_1,\ldots, dG_k$ are linearly independent at $x\in\M$,
 then so are they at $\phi_t(x)$.
\end{enumerate}
\end{lem}

\begin{proof}
For the proof of part~(i)
 we first transform \eqref{eqn:gsys0} to \eqref{eqn:aasys0} with $\ell=n-q$ and $m=q$,
 based on Proposition~\ref{LAJ}.
In particular, by Proposition~\ref{LAJ}(iia)
 $F\circ\varphi(I,\theta)$ depends only on $I$.
Let $\tilde{F}(I)=F\circ\varphi(I,\theta)$.
On the other hand, we use Lemma~1 in Section~1 of Chapter IV of \cite{K96}
 to show that $G(\varphi(I,\theta))$ depends only on $I$.
Let $\tilde{G}(I)=G(\varphi(I,\theta))$.
So we have
\[
G=\tilde{G}\circ\varphi^{-1}=\tilde{G}\circ\tilde{F}^{-1}\circ\tilde{F}\circ\varphi^{-1}
\]
and take $\psi=\tilde{G}\circ\tilde{F}^{-1}$ to obtain part~(i).

We turn to the proof of part~(ii).
Since the pull-back of  $G_j$ by $\phi_t$ satisfies $\phi_t^* G_j=G_j$ for $j=1,\ldots,k$, we have
\begin{align*}
&
\phi_t^*(dG_1\wedge \cdots \wedge dG_k)_x
=(\phi_t^*dG_1\wedge\cdots\wedge\phi_t^*dG_k)_x\\
&
=\left(d(\phi_t^*G_1)\wedge \cdots \wedge d(\phi_t^*G_k)\right)_x
=(dG_1\wedge \cdots \wedge dG_k)_x\neq0,
\end{align*}
where `$\wedge$' represents the wedge product.
On the other hand, by the definition of pull-back
\begin{align*}
\phi_t^*(dG_1\wedge \ldots \wedge dG_k)_x 
=(dG_1\wedge \ldots \wedge dG_k)_{\phi_t(x)}\circ(d\phi_t)_x.
\end{align*}
Since $\phi^t$ is a diffeomorphism,
 we see that $(d\phi_t)_x$ is regular.
 Hence,
\[
(dG_1\wedge \ldots \wedge dG_k)_{\phi_t(x)}\neq 0,
\]
which yields part (ii).
\end{proof}

%

Using Theorem~2.2 of  \cite{MY} on persistence of first integrals and Lemma~\ref{erg},
 we obtain the following result.

\begin{lem}\label{lem:key}
Let $I\in D_\R$.
Suppose that near $\epsilon=0$
 the perturbed system \eqref{eqn:p-gsys} has $n-q$ analytic first integrals
 $G_1^\epsilon, \ldots, G_{n-q}^\epsilon$
 near $\mathscr{T}_I=\{\gamma^{I}_\tau\mid\tau\in\Tset^q\}$ in $\M$
 such that they are functionally independent on $\mathscr{T}_I$
 and depend analytically on $\epsilon$.
Then $\I_F^{I}(\tau)$ must be identically zero.
\end{lem}

\begin{proof}
Assume that the hypothesis of the lemma holds.
Using Theorem~2.2 of \cite{MY}, we have
\begin{align*}
\I_{G^0}^{I}(\tau)&=\int_0^{T^{I}} dG^0(X^1)_{\gamma_\tau ^{I}(t)} dt=0.
\end{align*}
On the other hand, by Lemma~\ref{erg},
 $G^0:=(G_1^0, \ldots, G_{n-q}^0)$ is expressed as $G^0=\psi\circ F$
 for some analytic map $\psi:F(\mathcal{U})\to\Rset^{n-q}$,
 where $\mathcal{U}$ is a neighborhood of $F^{-1}(c)$.
Since $F$ is constant along the periodic orbit $\gamma^{I}_\tau(t)$,  we have
 $dG^0_{\gamma^{I}_\tau (t)}=d\psi_{F(\gamma^{I}_\tau (0))}dF_{\gamma^{I}_\tau(t)}$, so that
\begin{align*}
\I_{G^0}^{I}(\tau)
=\int_0^{T^{I}} d\psi_{F(\gamma_\tau ^{I}(0))}dF_{\gamma_\tau (t)}(X_{\gamma_\tau (t)})dt
=d\psi_{F(\gamma_\tau ^{I}(0))}\I_{F}^{I}(\tau).
\end{align*}
Since $dG^0$ and $dF$ have the maximum rank on $\gamma^{I}_\tau(t)$
 so that $\rank d\psi_{F(\gamma^{I}_\tau (0))}=n-q$,
 we obtain $\I_{F}^{I}(\tau)=0$ for any $\tau\in\Tset^q$.
\end{proof}

For the proof of Theorem~\ref{thm:main-fi}
 we also need the following result (see Appendix~A for its proof).

\begin{prop}\label{fcn-ind}
Let $k\le n-q$ be a positive integer.
Suppose that in a neighborhood of $F^{-1}(c)$
 the perturbed system \eqref{eqn:p-gsys} has $k$ first integrals
 that are analytic in $(x,\epsilon)$ near $\epsilon=0$.
If they are functionally independent for $\epsilon\neq 0$,
 then in a neighborhood of $F^{-1}(c)$
 there exist $k$ first integrals that are analytic in $(x,\epsilon)$ 
 and functionally independent near $\epsilon=0$.
\end{prop}

\begin{rmk}
\label{rmk:3a}
From Proposition~\ref{fcn-ind} we also see that
 the condition for $\D F_0(I)$ to have a maximum rank at a point $I_0\in\Rset^\ell$
 was unnecessary in Theorem~\ref{thm:Kozlov}.
\end{rmk}

\begin{proof}[Proof of Theorem~\ref{thm:main-fi}]
Suppose that $\I_{F}^{I}(\tau)$ is not identically zero for any $I\in D$.
Using Lemma~\ref{lem:key} for each $I\in D$, we see that
 the perturbed system \eqref{eqn:p-gsys} does not have $n-q$ analytic first integrals
 near $\mathscr{T}_I$ 
 such that they are functionally independent on $\mathscr{T}_I$ 
 and depend analytically on $\epsilon$ near $\epsilon=0$.

Additionally, suppose that
 there are $n-q$ analytic first integrals 
 such that they are functionally independent for $|\epsilon|\neq 0$ sufficiently small but not at $\epsilon=0$
 and depend analytically on $\epsilon$ near $\epsilon=0$.
Then by Proposition~\ref{fcn-ind}
 there exist $n-q$ analytic first integrals $G_1^\epsilon, \ldots, G_{n-q}^\epsilon$
 which are functionally independent and depend analytically on $\epsilon$ near $\epsilon=0$.
 Hence, $dG_1^0, \ldots, dG_{n-q}^0$ are linearly dependent on $\mathscr{T}_I$ for $I\in D$.
As in the proof of Lemma~\ref{erg} (i),
 we consider the transformed system \eqref{eqn:aasys0}
 and write $\tilde{G}_j(I)=G_j(\varphi(I,\theta))$ for $ j=1,\dots n-q$.
 Let $\tilde{G}(I)=(\tilde{G}_1(I),\ldots,\tilde{G}_{n-q}(I))$.
We see that the determinant of the Jacobi matrix of $\tilde{G}(I)$ is zero for $I\in D$,
 so that it is identically zero on $U$ since $D$ is a key set for $C^\omega(U)$.
 This yields a contradiction.
\end{proof}

\subsection{Proof of Theorem~\ref{thm:main-integ}}\label{pf:thm2}
We turn to the proof of Theorem~\ref{thm:main-integ}.
Henceforth, we assume that conditions~(A1)-(A4) and (K2) hold.

We begin with the following lemma.
Recall that $Y_l(x)$, $l=1,\ldots,q$, are commutative vector fields
 for the unperturbed system \eqref{eqn:gsys0}.

\begin{lem}\label{lem:cvf}
An analytic  vector field $Z(x)$ commutes with the vector field $X^0(x)$
if and only if it can be written as $Z(x)=\sum_{l=1}^q\rho_l(F(x))Y_l(x)$,
 where $\rho_l:\Rset^{n-q}\to\Rset$, $l=1,\ldots,n-q$, are analytic.
In particular, the vector field $X^0(x)$ has only $q$ commutative vector fields
which are linearly independent almost everywhere.
\end{lem}

\begin{proof}
By Proposition~\ref{LAJ}, we may consider \eqref{eqn:aasys0}.
So we prove that it is necessary and sufficient for an analytic vector field $Z(I, \theta)$
 to commute with the vector field of \eqref{eqn:aasys0} that it can be written as
\begin{equation}
Z(I,\theta)=\sum_{l=1}^q\rho_l(I) \frac{\partial}{\partial \theta_l}.
\label{eqn:cvf}
\end{equation}
The sufficiency is obvious.
The necessity follows from assumptions (K2) and (A4)
 by Lemma~1 in Section~3 of Chapter~IV of \cite{K96}.
\end{proof}

Using Theorem~3.5 of \cite{MY} on persistence of commutative vector fields and Lemma~\ref{lem:cvf},
 we obtain the following.

\begin{lem}\label{lem:key-cvf}
Let $I\in D_\R$.
Suppose that near $\epsilon=0$
 the perturbed system \eqref{eqn:p-gsys} has $q$ analytic commutative vector fields
 $Z_1^\epsilon(x), \ldots, Z_{q}^\epsilon(x)$
 near $\mathscr{T}_I=\{\gamma^{I}_\tau\mid\tau\in\Tset^q\}$ in $\M$
 such that they are linearly independent on $\mathscr{T}_I$
 and depend analytically on $\epsilon$.
Then $\I_F^{I}(\tau)$ must be constant.
\end{lem}

\begin{proof}
We first transform \eqref{eqn:p-gsys} to \eqref{eqn:aasys} with $\ell=n-q$ and $m=q$,
 as in the proof of Lemma~\ref{erg}.
Assume that the hypothesis of the lemma holds.
Then by Lemma~\ref{lem:cvf}, 
 $Z^0_j(I,\theta)$ has the form \eqref{eqn:cvf} with $\rho_l(I)=\rho_{jl}(I)$, $l=1,\ldots,q$,
  for $ j=1,\dots q$.
Let
\begin{align*}
\tilde{X}^1(I,\theta)=\sum_{j=1}^\ell h_j(I,\theta;0)\frac{\partial}{\partial I_j}
				+\sum_{k=1}^m g_k(I,\theta;0)\frac{\partial}{\partial \theta_k}
\end{align*}
where $h_j(I, \theta;\epsilon)$ (resp. $g_k(I, \theta;\epsilon)$) is $j$th (resp. $k$th) component of $h(I, \theta;\epsilon)$ (resp. $g(I, \theta;\epsilon)$) for $j=1,\ldots,\ell$ (for $k=1,\ldots,m$).
Using Theorem~3.5 of \cite{MY}, we have
\begin{align*}
0 =&  \int_0^{T^{I}} dI\left(\left[Z_j^0(I,\theta),
\tilde{X}^1(I,\theta)\right]\right)_{\theta=\omega(I)t+\tau}dt\\
=&\sum_{l=1}^q\rho_{jl}(I)\left(\int_0^{T^{I}} \frac{\partial h}{\partial \theta_l}(I, \omega(I)t+\tau; 0)dt\right)
=\sum_{l=1}^q\rho_{jl}(I)\frac{d\I_I^I}{d\tau_l}(\tau).
\end{align*}
Since $Z_1^0,\dots,Z_{q}^0$ are linearly independent on $\mathscr{T}_I$,
 the $q\times q$ matrix $(\rho_{jl})_{j,l=,1,\ldots,q}$ is invertible,
 so that by $n-q\le q$, $(d \I_I^I/d \tau)(\tau)=O$, i.e., $\I_I^I(\tau)$ is constant.

On the other hand, by Lemma~\ref{erg},
 there exists an analytic map $\psi:F(\mathcal{U})\to\Rset^k$ such that $I=\psi\circ F(x)$.
As in the proof of Lemma~\ref{lem:key},
 we show that $\rank d\psi_{F(\gamma^I(0))}=n-q$
 and $\I_I^I(\tau)=d\psi_{F(\gamma^I(0))}\I_F^I(\tau)$.
Hence, $\I_F^I(\tau)$ is also constant.
\end{proof}

\begin{proof}[Proof of Theorem~\ref{thm:main-integ}]
Suppose that $\I_{F}^{I}(\tau)$ is not constant for any $I\in D$.
From Theorem~\ref{thm:main-fi} we see that there exist only $n-q-1$ first integrals at most
 such that they are functionally independent and 
 depend analytically on $\epsilon$ near $\epsilon=0$.
 
Assume that there exist $q$ analytic commutative vector fields
 $Z_1^\epsilon, \ldots,Z_q^\epsilon$
 such that for $\epsilon=0$ they are linearly independent almost everywhere
 and depend analytically on $\epsilon$.
Applying Lemma~\ref{lem:key-cvf} for each $I\in D$, we see that
 the perturbed system \eqref{eqn:p-gsys} does not have $q$ analytic commutative vector fields
 near $\mathscr{T}_I$ 
 such that they are linearly independent on $\mathscr{T}_I$ 
 and depend analytically on $\epsilon$ near $\epsilon=0$.
Hence, $Z_1^\epsilon,\ldots,Z_{q}^\epsilon$ are linearly dependent
 on $\mathscr{T}_I$ for $I\in D$ at $\epsilon=0$
 if they depend analytically on $\epsilon$ near $\epsilon=0$.
As in the proof of Lemma~\ref{lem:key-cvf},
 we consider the transformed system \eqref{eqn:aasys0}
 and use Lemma~\ref{lem:cvf}
 to write $Z^0_j(I,\theta)$ in the form \eqref{eqn:cvf} with $\rho_l(I)=\rho_{jl}(I)$, $l=1,\ldots,q$,
 for $ j=1,\dots q$.
We see that the determinant of the matrix $(\rho_{jl}(I))_{j,l=,1,\ldots,q}$ is zero for $I\in D$,
 so that it is identically zero on $U$ since $D$ is a key set for $C^\omega(U)$.
This yields a contradiction.
So we obtain the desired result.
\end{proof}


\section{Consequences of the Theory to \eqref{eqn:aasys}}\label{inAA}
In this section, we consider nearly integrable systems of the form \eqref{eqn:aasys}
 written in the action-angle coordinates
 and describe consequences of Theorems~\ref{thm:main-fi} and \ref{thm:main-integ} to it.
The unperturbed system~\eqref{eqn:aasys0} is $(m,\ell)$-integrable
 in the Bogoyavlenskij sense
 and has $\ell$ first integrals $I_1,\ldots,I_{\ell}$ and $m$ commutative vector fields
$\omega(I)\frac{\partial}{\partial \theta_1}, \frac{\partial}{\partial \theta_2}, \ldots, \frac{\partial}{\partial \theta_m}$.
Thus, conditions~(A1) and (A2) with $n=\ell+m$ and $q=m$ already hold.
In particular, the level set of $I=c$ given by $\{c\}\times\Tset^m$
 is connected and compact.
Take some $I_0\in\Rset^\ell$
 and let $U$ be its neighborhood in $\Rset^\ell$, as in the preceding sections.

We first discuss consequences of Theorem~\ref{thm:main-fi} to \eqref{eqn:aasys}
 and assume that conditions~(K2) and (A3) with $n=\ell+m$ and $q=m$ hold.
For $I\in D_\R$ the unperturbed system \eqref{eqn:aasys0}
 has an $m$-parameter family of periodic orbits given by \eqref{eqn:po} with $q=m$.
The integrals given by \eqref{eqn:i} for the $\ell$ first integrals $I=(I_1,\ldots,I_{\ell})$ become
\begin{align*}
\I_I^I(\tau)
 =\int_0^{T^{\bar{I}}} h(I ,\omega(I)t+\tau;0)dt,
\end{align*}
where $\tau\in\Tset^m$.

Assume that $m>1$.
Using the Fourier expansion of $h(I,\theta;0)$ given in \eqref{eqn:hath},
 we rewrite the above integral as
\begin{align}
\I_I^I(\tau)
=&\int_0^{T^{\bar{I}}}\sum_{r\in \Zset^m}
 \hat{h}_r(I)\exp(i r\cdot(\omega(I)t+\tau))d t
= T^I\sum_{r\in \Lambda_I}\hat{h}_r(I)e^{i r\cdot\tau},
\label{eqn:int-FT}
\end{align}
where $\Lambda_I=\{r\in\Zset^m\mid r\cdot\omega(I)=0\}$.
Applying Theorem~\ref{thm:main-fi},
 we obtain the following result for \eqref{eqn:aasys}.
Recall that $\hat{h}_r(I)$, $r\in\Zset^m$, represent the Fourier coefficients of $h(I,\theta;0)$
 (see Eq.~\eqref{eqn:hath}).
 
\begin{thm}\label{thm:4a}
Let $m>1$, and suppose that assumptions~(K2) and (A3) with $n=\ell+m$ and $q=m$ hold.
If there exists a key set $D\subset D_\R$ for $C^\omega(U)$
 such that $\hat{h}_r(I)\neq 0$ for some $r\in\Lambda_I$ with $I\in D$,
 then the perturbed system \eqref{eqn:aasys} does not have $\ell$ real-analytic first integrals
 in a neighborhood of the level set $\{c\}\times\Tset^m$ near $\epsilon=0$
 such that they are functionally independent for $|\epsilon|\neq 0$
 and depend analytically on $\epsilon$.
\end{thm}

\begin{rmk}\
\label{rmk:4a}
\begin{itemize}
\setlength{\leftskip}{-1.8em}
\item[\rm(i)]
From the proof given in \cite{K96}
 we see that the conclusion of Theorem~\ref{thm:Kozlov} also holds
 even if the zero vector is taken as one of $r_j\in\Zset^m$, $j=1,\ldots,\ell-s$,
 in the definition of a Poincar\'e set.
This fact was overlooked in \cite{K96}.
\item[\rm(ii)]
If a Poincar\'e set $\P_{\ell-1}\subset U$ modified as stated in part~(i)
 is a key set for $C^\omega(U)$, then condition~(A3) holds.
Moreover, there exists a key set $D\subset D_\R$ for $C^\omega(U)$
 such that $\hat{h}_r(I)\neq 0$ with some $r\in\Lambda_I$ for $I\in D$
 if and only if such a Poincar\'e set $\P_{\ell-1}\subset U$ is a key set for $C^\omega(U)$.
\item[\rm(iii)] The hypothesis of Theorem~\ref{thm:4a} holds
if both of $\omega(I)$ and $\hat{h}_0(I)$ are not identically zero in $U$.
\item[\rm(iv)]
If the system \eqref{eqn:aasys} is Hamiltonian, then $\hat{h}_0(I)\equiv 0$.
\item[\rm(v)]
From Theorem~2.2 of \cite{MY} and Eq.~\eqref{eqn:int-FT}
 we see that the first integrals $I_1,\ldots, I_m$ do not persist in \eqref{eqn:aasys}
 near the resonant torus $\{I\}\times\Tset^m$
 if $\hat{h}_r(I)\neq 0$ for some $r\in \Lambda_I$.
\end{itemize}
\end{rmk}

Let $m=1$ and assume that $\omega(I)\neq 0$.
Then the integral \eqref{eqn:int-FT} becomes
\begin{equation}
\I_{I}^{I}(\tau)=\int_0^{2\pi/\omega(I)} h\left( I ,\omega(I)t+\tau ;0\right) dt
=\frac{2\pi \hat{h}_0(I)}{\omega(I)}
\label{eqn:int-FT1}
\end{equation}
since $\Lambda_I=\{0\}\subset\Zset$.
Noting that assumptions (K2) and (A3) hold if $\omega(I)\neq 0$ for some $I\in U$,
 we obtain the following.

\begin{thm}\label{thm:4b}
Let $m=1$.
If 
 $\omega(I)\neq 0$ and $\hat{h}_0(I)\neq 0$ for some $I\in U$,
 then the perturbed system \eqref{eqn:aasys} does not have $\ell$ real-analytic first integrals
 in a neighborhood of $\{c\}\times\Sset^1$ near $\epsilon=0$
 such that they are functionally independent for $|\epsilon|\neq 0$
 and depend analytically on $\epsilon$.
\end{thm}

Assuming the existence of $\ell-1$ functionally independent first integrals
 in the perturbed system \eqref{eqn:aasys}
 and taking Remarks~\ref{rmk:4a}(i) and (ii) into account,
 we obtain  the same result as Theorems~\ref{thm:4a} and \ref{thm:4b} from Theorem~\ref{thm:Kozlov}
(see also Remark~\ref{rmk:3a}).
Moreover, when the existence of such only $s\ (<\ell-1)$ first integrals is assumed,
 Theorem~\ref{thm:Kozlov} guarantees the nonexistence of no additional first integral
 if a Poincar\'e set $\P_s$ modified in Remark~\ref{rmk:4a}(i) is a key set for $C^\omega(U)$,
 in particular $\hat{h}_{r_j}(I)$, $j=1,\ldots,\ell-s\ (>1)$, are linearly independent.
Note that such a Poincar\'e set $\P_s$ does not exist when $m=1$ and $s<\ell-1$.

We next apply Theorem~\ref{thm:main-integ} to \eqref{eqn:aasys}.
When $m=1$, the integral $\I_I^I(\tau)$ is constant by \eqref{eqn:int-FT1},
 so that Theorem~\ref{thm:main-integ}  does not apply.
Thus, we obtain the following result.
 
\begin{thm}\label{thm:4c}
Let $m>1$, and suppose that assumptions~(K2), (A3) and (A4) hold.
If there exists a key set $D\subset D_\R$ for $C^\omega(U)$
 such that $\hat{h}_r(I)\neq 0$ for some $r\in\Lambda_I\setminus\{0\}$ with $I\in D$,
 then for $|\epsilon|\neq 0$ sufficiently small
 the perturbed system \eqref{eqn:aasys} is not real-analytically integrable
 in the meaning of Theorem~\ref{thm:main-integ} near $\{c\}\times\Tset^m$.
\end{thm}


\section{Relationships with the Melnikov Methods} 
\label{Melnikov}
In this section, we discuss relationships of our main results in Section~2
 with the subharmonic and homoclinic Melnikov methods 
 for time-periodic perturbations of single-degree-of-freedom Hamiltonian systems.
See \cite{GH83,M63,W90,Y96} for the details of the Melnikov methods.
A concise review of the methods was also given in Section~4.1 of \cite{MY}.

Consider systems of the form
\begin{align}\label{mel}
\dot{x}=J\D H(x)+\varepsilon u(x,\nu t),\quad
x\in\Rset^2,
\end{align}
where $\epsilon$ is a small parameter as in the preceding sections, $\nu>0$ is a constant,
 $H:\mathbb{R}^2\to\mathbb{R}$ and $u:\Rset^2\times\Sset\to\Rset^2$ are analytic,
 and $J$ is the $2\times 2$ symplectic matrix,
\[
J=
\begin{pmatrix}
0 & 1\\
-1 & 0
\end{pmatrix}.
\]
Equation~\eqref{mel} represents a time-periodic perturbation
 of the single-degree-of-freedom Hamiltonian system
\begin{align}\label{mel0}
\dot{x}=J\D H(x),
\end{align}
with the Hamiltonian $H(x)$.
Letting $\phi=\nu t\mod 2\pi$ such that $\phi\in\Sset^1$,
 we rewrite \eqref{mel} as an autonomous system,
\begin{equation}\label{mel-aut}
\dot{x}=J\D H(x)+\varepsilon u(x,\phi),\quad
\dot{\phi}=\nu.
\end{equation}
We easily see that assumptions~(A1) and (A2) hold in \eqref{mel-aut} with $\epsilon=0$:
 $H(x)$ is a first integral and $(0,1)\in\Rset^2\times\Rset$ is a commutative vector field.
We make the following assumptions on the unperturbed system \eqref{mel0}:
\begin {enumerate}
\setlength{\leftskip}{-0.6em}
\item[\bf(M1)]
There exists a one-parameter family of periodic orbits $x^{\alpha}(t)$
 with period $\hat{T}^{\alpha}>0$, $\alpha\in(\alpha_1,\alpha_2)$, for some $\alpha_1<\alpha_2$.
Moreover, $\hat{T}^{\alpha}$ is not constant as a function of $\alpha$.
 \item[\bf(M2)]
$x^{\alpha}(t)$ is analytic with respect to $\alpha\in(\alpha_1,\alpha_2)$.
\end{enumerate}
Note that in (M1) $x^{\alpha}(t)$ is automatically analytic with respect to $t$
 since the vector field of \eqref{mel0} is analytic.

We assume that at $\alpha=\alpha^{l/n}$
\begin{equation}
\frac{2\pi}{\hat{T}^\alpha}=\frac{n}{l}\nu,
\label{eqn:res4}
\end{equation}
where $l$ and $n$ are relatively prime integers.
We define the \emph{subharmonic Melnikov function} as
\begin{equation}
M^{l/n}(\phi)=\int_0^{2\pi l/\nu}\D H(x^{\alpha}(t))\cdot u(x^\alpha(t),\nu t+\phi)\d t,
\label{eqn:subM}
\end{equation}
where $\alpha=\alpha^{l/n}$.
Let $T^\alpha=n\hat{T}^\alpha=2\pi l/\nu$ for $\alpha=\alpha^{l/n}$.
If $M^{l/n}(\phi)$ has a simple zero at $\phi=\phi_0$
 and $d\hat{T}^\alpha/d\alpha\neq 0$ at $\alpha=\alpha^{l/n}$,
 then for $|\epsilon|>0$ sufficiently small
 there exists a $T^\alpha$-periodic orbit
 near $(x,\phi)=(x^\alpha(t),\nu t+\phi_0)$ in \eqref{mel-aut}.
See Theorem~3.1 of \cite{Y96}.
A similar result is also found in \cite{GH83,W90}.
The stability of the periodic orbit can also be determined easily \cite{Y96}.
Moreover, several bifurcations of periodic orbits
 when $d\hat{T}^\alpha/d\alpha\neq 0$ or not
 were discussed in \cite{Y96,Y02,Y03}.

On the other hand, since it is a single-degree-of-freedom Hamiltonian system,
 the unperturbed system \eqref{mel0} is integrable,
 so that it can be transformed into the form \eqref{eqn:aasys0} with $\ell,m=1$.
So the perturbed system \eqref{mel-aut} is transformed into the form \eqref{eqn:aasys}
 with $\ell=1$ and $m=2$.
Here we take $I=\alpha$ unlike \cite{Y96,Y21a},
 and have $\omega(I)=(\Omega(I),\nu)$, where
\[
\Omega(\alpha)=\frac{2\pi}{\hat{T}^\alpha}.
\]
We remark that the transformed system is not Hamiltonian even when $\epsilon=0$,
 unlike \cite{Y96,Y21a}.
Choose a point $\alpha=\alpha_0\in(\alpha_1,\alpha_2)$
 such that $d\hat{T}^\alpha/d\alpha\neq 0$,
 and let $U$ be a neighborhood of $\alpha_0$.
We see that assumptions~(K2) and (A3) hold for
\[
D_\R=\{\alpha^{l/n}\mid l,n\in\Nset\}\cap U.
\]

Let $\alpha=\alpha^{l/n}$ and let $\gamma_\tau^\alpha(t)=(x^\alpha(t+\tau_1),\nu (t+\tau_1)+\tau_2)$.
We see that $\gamma_\tau^\alpha(t)$ is a $T^\alpha$-periodic orbit
 in \eqref{mel-aut} with $\epsilon=0$.
Note that $\gamma_\tau^\alpha(t)$ is essentially parameterized by a single parameter,
 say $\phi:=\nu \tau_1+\tau_2$.
So we write
 $\gamma_\phi^\alpha(t)=(x^\alpha(t),\nu t+\phi)$. 
The integral \eqref{eqn:i} for $H(x)$ along $\gamma_\phi^\alpha(t)$ 
 becomes 
\begin{align}
\I_H^\alpha(\phi)
=&\int_0^{2\pi l/\nu}\D H(x^\alpha(t))\cdot u(x^\alpha(t),\nu t+\phi)dt
=M^{l/n}(\phi)
\end{align}
by \eqref{eqn:subM}.
As stated above,
 if $M^{l/n}(\phi)$ has a simple zero at $\phi=\phi_0$,
  then there exists a $T^\alpha$-periodic orbit near $\gamma_{\phi_0}^\alpha(t)$.
Applying Theorems~\ref{thm:main-fi} and \ref{thm:main-integ}, we have the following two results.

\begin{thm}\label{thm:sub-mel1}
Suppose that there exists a key set $D\subset D_\R$ for $C^\omega(U)$ 
 such that $M^{l/n}(\phi)$ is not identically zero for $\alpha^{l/n}\in D$.
Then for $|\epsilon|\neq 0$ sufficiently small
 the system \eqref{mel-aut} has no real-analytic first integral
 in a neighborhood of $\{x^{\alpha_0}(t)\mid t\in[0,\hat{T}^{\alpha_0})\}\times\Sset^1$
 such that it depends analytically on $\epsilon$ near $\epsilon=0$.
\end{thm}

\begin{thm}\label{thm:sub-mel2}
Suppose that there exists a key set $D\subset D_\R$ for $C^\omega(U)$ 
 such that $M^{l/n}(\phi)$ is not constant for $\alpha^{l/n}\in D$.
Then for $|\epsilon|\neq 0$ sufficiently small
 the system \eqref{mel-aut} is not real-analytically integrable
 in the meaning of Theorem~\ref{thm:main-integ}
 in a neighborhood of $\{x^{\alpha_0}(t)\mid t\in[0,\hat{T}^{\alpha_0})\}\times\Sset^1$.
\end{thm}

\begin{rmk}\
\label{rmk:5a}
\begin{itemize}
\setlength{\leftskip}{-1.8em}
\item[\rm(i)]
If $D$ has an accumulation point,
 then it becomes a key set for $C^\omega(U)$.
\item[\rm(ii)]
If the system \eqref{mel} is Hamiltonian,
 then the hypotheses of Theorems~\ref{thm:sub-mel1} and \ref{thm:sub-mel2}
 are equivalent to the condition that $M^{l/n}(\phi)$ has a simple zero.
Actually, letting
\[
u(x,\phi)=J\D_x H^1(x,\phi)=J\sum_{r\in\Zset}\D\hat{H}_r^1(x)e^{ir\phi},
\]
we have
\begin{align*}
M^{l/n}(\phi)=&\int_0^{2\pi l/\nu}\D H(x^\alpha(t))\cdot J\D_x H^1(x^\alpha(t),\nu t+\phi)dt\\
=&\sum_{r\in\Zset}e^{ir\phi}\int_0^{2\pi l/\nu}\D H(x^\alpha(t))\cdot J\D\hat{H}_r^1(x^\alpha(t))e^{ir\nu t}d t
\end{align*}
and
\begin{align*}
&
\int_0^{2\pi l/\nu}\D H(x^\alpha(t))\cdot J\D\hat{H}_0^1(x^\alpha(t))d t\\
&
=-\int_0^{2\pi l/\nu}\D\hat{H}_0^1(x^\alpha(t))\cdot J\D H(x^\alpha(t))d t\\
&
=-\int_0^{2\pi l/\nu}\D\hat{H}_0^1(x^\alpha(t))\cdot \dot{x}^\alpha(t)d t=0,
\end{align*}
where $\hat{H}_r^1(x)$, $r\in\Zset$, represent the Fourier coefficients of $H^1(x,\phi)$.
Thus, we obtain the claim.
\end{itemize}
\end{rmk}


\begin{figure}
\includegraphics[scale=0.9,bb=0 0 179 122]{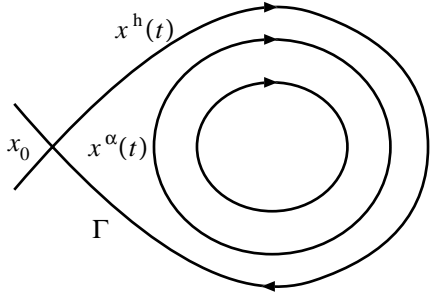}
\caption{Assumption~(M3).\label{fig:M3}}
\end{figure}

We additionally assume the following on the unperturbed system \eqref{mel0}:
\begin {enumerate}
\setlength{\leftskip}{-0.6em}
\item[\bf(M3)]
There exists a hyperbolic saddle $x_0$ with a homoclinic orbit $x^\h(t)$ such that
\[
\lim_{\alpha\to\alpha_2}\sup_{t\in\Rset}d(x^\alpha(t),\Gamma)=0,
\]
where $\Gamma=\{x^\h(t)\mid t\in\Rset\}\cup\{x_0\}$
 and $d(x,\Gamma)=\inf_{y\in\Gamma}|x-y|$.
See Fig.~\ref{fig:M3}.
\end{enumerate}
We define the \emph{homoclinic Melnikov function} as 
\begin{equation}
M(\phi)=\int_{-\infty}^{\infty}\D H(x^{\h}(t))\cdot u(x^{\h}(t),t+\phi)dt.
\label{eqn:homM}
\end{equation}
If $M(\phi)$ has a simple zero,
 then for $|\epsilon|>0$ sufficiently small
 there exist transverse homoclinic orbits to a periodic orbit near $\{x_0\}\times\Sset^1$
 in \eqref{mel-aut} \cite{GH83,M63,W90}.
The existence of such transverse homoclinic orbits
 implies that the system \eqref{mel-aut} exhibits chaotic motions
 by the Smale-Birkhoff theorem \cite{GH83,W90}
 and has no resl-analytic (additional) first integral (see, e.g., Chapter III of \cite{M73}).
We easily show that
\begin{equation}
\lim_{l\to \infty} M^{l/1}(\phi)=M(\phi)
\label{eqn:limM}
\end{equation}
for each $\phi\in \Sset^1$ (see Theorem~4.6.4 of \cite{GH83}).
Let $U$ be a neighborhood of $\alpha=\alpha_2$.
It follows from \eqref{eqn:limM} that
 if $M(\phi)$ is not identically zero or constant,
 then for $l>0$ sufficiently large neither is $M^{l/1}(\phi)$.
Let $\hat{U}\subset\Rset^2$ be a region such that $\partial\hat{U}\supset \Gamma$
and $\hat{U}\supset\{x^\alpha(t)\mid t\in[0,\hat{T}^\alpha)\}$
for some $\alpha\in(\alpha_1,\alpha_2)$.
We obtain the following from Theorems~\ref{thm:sub-mel1} and \ref{thm:sub-mel2}.

\begin{thm}\label{thm:hom-mel1}
Suppose that $M(\phi)$ is not identically zero
Then for $|\epsilon|\neq 0$ sufficiently small
 the system \eqref{mel-aut} has no real-analytic first integral
 in $\hat{U}\times\Sset^1$
 such that it depends analytically on $\epsilon$ near $\epsilon=0$.
\end{thm}

\begin{thm}\label{thm:hom-mel2}
Suppose that $M(\phi)$ is not constant.
Then for $|\epsilon|\neq 0$ sufficiently small
 the system \eqref{mel-aut} is not real-analytically integrable
 in the meaning of Theorem~\ref{thm:main-integ}
 in $\hat{U}\times\Sset^1$. 
\end{thm}

\begin{rmk}\
\label{rmk:5b}
\begin{enumerate}
\setlength{\leftskip}{-1.8em}
\item[(i)]
Theorems~\ref{thm:hom-mel1} and \ref{thm:hom-mel2}, respectively,
 mean that the system \eqref{mel-aut} has no first integral and is nonintegrable
 even if the Melnikov function $M(\phi)$ does not have a simple zero,
 i.e., there may exist no transverse homoclinic orbit to the periodic orbit in \eqref{mel-aut},
 but it is not identically zero and constant.
See Section~\ref{Duf}.
\item[(ii)]
As in Remark~\ref{rmk:5a}(ii), if the system \eqref{mel} is Hamiltonian,
 then the hypotheses of Theorems~\ref{thm:hom-mel1} and \ref{thm:hom-mel2}
 are equivalent to the condition that $M^{l/n}(\phi)$ has a simple zero.
\item[(iii)]
In the statements of Theorems~\ref{thm:hom-mel1} and \ref{thm:hom-mel2},
 the region $\hat{U}\times\Sset^1$ may be replaced with a neighborhood of $\Gamma\times\Sset^1$
 although they are weakened.
 \end{enumerate}
\end{rmk}


\section{Examples}\label{examples}

We now illustrate the above theory for four examples:
 Simple pendulum with a constant torque, 
 second-order coupled oscillators and the periodically forced Duffing oscillator \cite{GH83,H79,W90}.

\subsection{Simple pendulum with a constant torque}
Consider a simple pendulum with a constant torque:
\begin{align}
\dot{I}=\epsilon (\beta\sin\theta +1),\quad \dot{\theta}=I, \quad (I,\theta)\in \Rset\times \Tset\label{simp-torque}
\end{align}
where $\beta\in\Rset$ is a constant.
Equation~\eqref{simp-torque} is of the form \eqref{eqn:aasys} with $m=\ell=1$.
Using Theorem~\ref{thm:4b}, we obtain the following.

\begin{prop}
The system \eqref{simp-torque} has no real-analytic first integral
 depending analytically on $\epsilon$ near $\epsilon=0$.
\end{prop}

\begin{rmk}\
\begin{enumerate}
\setlength{\leftskip}{-2em}
\item[(i)]
The system \eqref{simp-torque} has the first integral
\[
F(I, \theta; \epsilon)=\frac{1}{2}I^2+\epsilon (\beta\cos\theta-\theta)
\]
and is $(1,1)$-integrable as a system on $\Rset\times\Rset$, although $F(I,\theta; \epsilon)$ is not even a function on $\Rset\times\Sset^1$.
\item[(ii)]
Let $\beta=0$.
Then the system \eqref{simp-torque} is $(2,0)$-integrable when $\epsilon\neq0$,
 where the vector fields $\epsilon\frac{\partial}{\partial I}+I\frac{\partial}{\partial \theta}$ and  $\frac{\partial}{\partial \theta}$ are commutative and linearly independent.
However, when $\epsilon=0$, the two vector fields are linearly dependent
 and Eq.~\eqref{simp-torque} is not $(2,0)$-integrable.
See also Remark~\ref{rmk:2a}(ii).
\end{enumerate}
\end{rmk}

\subsection{Second-order coupled oscillators}
Consider
\begin{equation}
\begin{split}
\dot{I}_j=&\epsilon\left(-\delta I_j+\Omega_j
 +\sum_{i=1}^\ell \sum_{k\in\Nset^2}  a_{k}\sin(k_1\theta_j-k_2\theta_i)\right),\\
\dot{\theta}_j=&I_j,\quad j=1, \ldots, \ell,
\end{split}
\label{eqn:kuramoto-sys}
\end{equation}
where $\delta,\Omega_j\ge 0$, $j=1,\ldots,\ell$, and $a_{k}$, $k=(k_1,k_2)\in\Nset^2$, are constants
 such that $|a_k|\le Me^{-(k_1+k_2)\delta}$ for some $M,\delta>0$.
We see by the remark after Lemma~2 in Section~12 of Chapter~3 in \cite{A88}
 that the vector field of \eqref{eqn:kuramoto-sys} is analytic.
Equation~\eqref{eqn:kuramoto-sys} has the form \eqref{eqn:aasys} with $m=\ell$
 and is rewritten in a system of second-order differential equations as
\begin{align*}
\ddot{\theta}_j+\epsilon\delta\dot{\theta}
=\epsilon\left(\Omega_j+\sum_{i=1}^\ell \sum_{k\in \Nset^2}  a_{k}\sin(k_1\theta_j-k_2\theta_i)\right),
\quad j=1, \ldots, \ell,
\end{align*}
which reduces to the \emph{second-order Kuramoto model} \cite{RPJK16}
 when $a_k\neq 0$ for $k=(1,1)$ and $a_k=0$ for $k\neq (1,1)$.
Obviously, assumptions~(K2), (A3) and (A4) hold.
Using Theorems~\ref{thm:4a} and \ref{thm:4c}, we obtain the following.

\begin{prop}
The following statements hold for \eqref{eqn:kuramoto-sys}:
\begin{itemize}
\setlength{\leftskip}{-2em}
\item[(i)]
If one of $\delta$ and $\Omega_j$, $j=1, \ldots, \ell$, is nonzero at least, 
 then the system \eqref{eqn:kuramoto-sys} does not have $\ell$ real-analytic first integrals near $\epsilon=0$
 such that they are functionally independent for $|\epsilon|\neq 0$ and depend analytically on $\epsilon$;
\item[(ii)]
If $K_1=\{k_1/k_2\mid a_k,k_2\neq 0\}$ or $K_2=\{k_2/k_1\mid a_k,k_1\neq 0\}$ has an accumulation point,
then for $|\epsilon|\neq 0$ sufficiently small
 the system \eqref{eqn:kuramoto-sys} is not real-analytically integrable
 in the meaning of Theorem~\ref{thm:main-integ}.
\end{itemize}
\end{prop}
\begin{proof}
Part (i) immediately follows from Theorem~\ref{thm:4a} and Remark~\ref{rmk:4a}(iii)
 since $\hat{h}_0(I)$ is not identically zero
 if one of $\delta$ and $\Omega_j$, $j=1, \ldots, \ell$, is nonzero at least.

We turn to the proof of part (ii).
Let $D_1=\{I\in\Rset^\ell\mid k_1I_j-k_2I_i=0,\  i,j=1,\dots,\ell, \  k_1/k_2\in K_1\}$
 and $D_2=\{I\in\Rset^\ell\mid k_1I_j-k_2I_i=0,\  i,j=1,\dots,\ell, \  k_2/k_1\in K_2\}$.
If $K_1$ (resp. $K_2$) has an accumulation point,
 then $D_1$ (resp. $D_2$) is a key set for $C^\omega(\Rset^\ell)$.
Their claim is shown as follows.
Assume that $K_1$ has an accumulation point. Let $f(I)\in C^\omega(\Rset^\ell)$ be an analytic function which vanishes on $D_1$, and take a line $L_b:=\{I\in\Rset^\ell\mid (I_1,\ldots, I_{\ell-1})=b\}$ for $b=(b_1, \ldots,b_{\ell-1})\in\Rset^{\ell-1}$ fixed . 
Then $(b_1, \ldots,b_{\ell-1}, k_1b_1/k_2) \in L_b\cap D_1$ for all $k_1/k_2\in K_1$, so that $f(I)$ is identically zero on $L_b$.
This means that $f(I)$ is identically zero in $\Rset^{\ell}$, and consequently $D_1$ is a key set for $C^\omega(\Rset^\ell)$.
Similarly, we see that the claim is true for $K_2$ and $D_2$.
Applying Theorem~\ref{thm:4c}, we obtain the desired result.
\end{proof}

\subsection{Periodically forced Duffing oscillator}\label{Duf}

\begin{figure}[t]
\includegraphics[scale=0.27,bb=0 0 690 572]{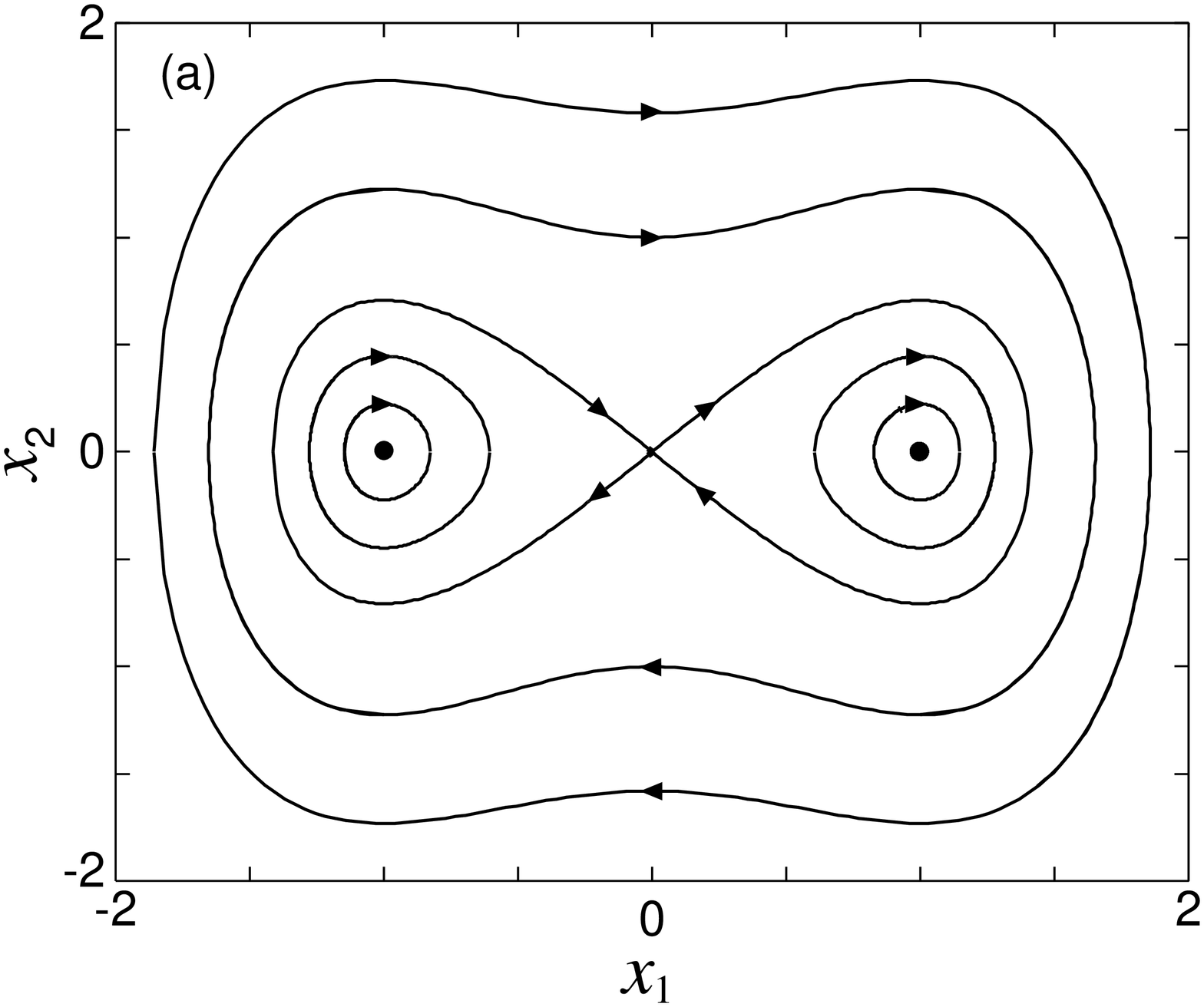}\quad
\includegraphics[scale=0.28,bb=0 0 559 551]{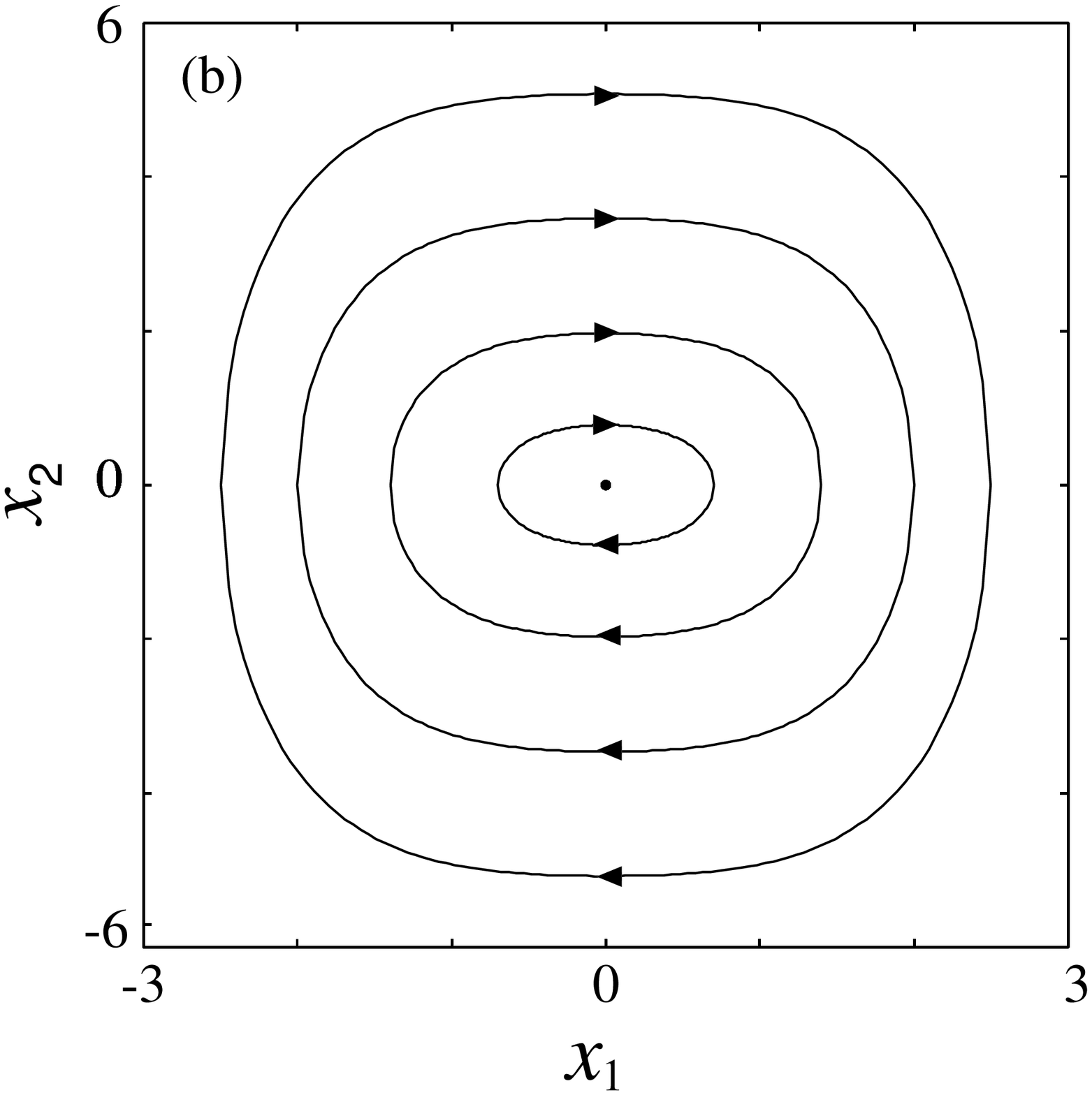}
\caption{Phase portraits of \eqref{Duffing} with $\epsilon=0$: (a) $a=1$; (b) $a=-1$.
\label{fig:Duffing}}
\end{figure}

Consider the periodically forced Duffing oscillator
\begin{equation}\label{Duffing}
\dot{x}_1=x_2,\quad
\dot{x}_2=ax_1-x_1^3+\epsilon(\beta\cos\nu t-\delta x_2),
\end{equation}
where $\nu>0$ and $\beta,\delta\ge 0$ are constants, and $a=-1$ or $1$.
The system~\eqref{Duffing} has the form \eqref{mel} with
\[
H=-\frac{1}{2} ax_1^2+\frac{1}{4}x_1^4+\frac{1}{2}x_2^2
\]
and the autonomous system \eqref{mel-aut} becomes
\begin{align}\label{DuffingAut}
\dot{x}_1=x_2,\quad
\dot{x}_2=ax_1-x_1^3+\epsilon(\beta\cos\theta-\delta x_2),\quad
\dot{\phi}=\nu,
\end{align}
where $(x,\phi)\in\Rset^2\times \Sset$.
See Fig.~\ref{fig:Duffing} for the phase portraits of \eqref{Duffing} with $\epsilon=0$.

We begin with the case of  $a=1$.
When $\epsilon=0$, in the phase plane there exist a pair of homoclinic orbits
\[
x^\h_\pm(t)=(\pm\sqrt{2}\sech t, \mp\sqrt{2}\sech t\,\tanh t),
\]
a pair of one-parameter families of periodic orbits
\begin{align*}
x^k_{\pm}(t)
 =&\biggl(\pm\frac{\sqrt{2}}{\sqrt{2-k^2}}\dn\left(\frac{t}{\sqrt{2-k^2}}\right),\\
& \quad\mp\frac{\sqrt{2}k^2}{2-k^2}\sn\left(\frac{t}{\sqrt{2-k^2}}\right)
\cn\left(\frac{t}{\sqrt{2-k^2}}\right)\biggr),\quad
k\in(0,1),
\end{align*}
inside each of them, and a one-parameter periodic orbits
\begin{align*}
\tilde{x}^k(t)
 =&\biggl(\frac{\sqrt{2}k}{\sqrt{2k^2-1}}\cn\left(\frac{t}{\sqrt{2k^2-1}}\right),\\
& \quad -\frac{\sqrt{2}k}{2k^2-1}\sn\left(\frac{t}{\sqrt{2k^2-1}}\right) 
 \dn\left(\frac{t}{\sqrt{2k^2-1}}\right)\biggr),\quad
k\in\bigl(1/\sqrt{2},1\bigr),
\end{align*}
outside of them, as shown in Fig.~\ref{fig:Duffing}(a),
 where $\sn$, $\cn$ and $\dn$ represent the Jacobi elliptic functions with the elliptic modulus $k$.
The periods of $x^k_{\pm}(t)$ and $\tilde{x}^k(t)$ are given
 by $\hat{T}^k=2K(k)\sqrt{2-k^2}$ and $\tilde{T}^k=4K(k)\sqrt{2k^2-1}$, respectively,
 where $K(k)$ is the complete elliptic integral of the first kind.
Note that $x^k_{\pm}(t)$ approaches $x^\h_{\pm}(t)$ as $k\to 1$.
See \cite{GH83,W90}.
See also \cite{BF54} for general information on elliptic functions.

Assume that the resonance conditions
\begin{equation}
l\hat{T}^k=\frac{2\pi n}{\nu},\quad\mbox{i.e.,}\quad
\nu=\frac{2\pi n}{2lK(k)\sqrt{2-k^2}},
\label{eqn:resk}
\end{equation}
and
\begin{equation}
l\tilde{T}^k=\frac{2\pi n}{\nu},\quad\mbox{i.e.,}\quad
\nu=\frac{2\pi n}{4lK(k)\sqrt{2k^2-1}},
\label{eqn:tresk}
\end{equation}
hold for $x_\pm^k(t)$ and $\tilde{x}^k(t)$, respectively,
 with $l,n>0$ relatively prime integers.
Then the subharmonic Melnikov function \eqref{eqn:subM}
 for $x_\pm^k(t)$ and $\tilde{x}^k(t)$ are
\[
M_\pm^{n/l}(\tau)=-\delta J_1(k,l)\pm\beta J_2(k,n,l)\sin\tau
\]
and
\[
\tilde{M}^{n/l}(\tau)=-\delta\tilde{J}_1(k,l)+\beta\tilde{J}_2(k,n,l)\sin\tau,
\]
respectively, where
\begin{align*}
&
J_1(k,l)=\frac{4l[(2-k^2)E(k)-2k'^2K(k)]}{3(2-k^2)^{3/2}},\\
&
J_2(k,n,l)=
\begin{cases}
\sqrt{2}\pi \nu\sech\left(\displaystyle\frac{n\pi K(k')}{K(k)}\right) & \mbox{(for $l=1$)};\\
0\quad & \mbox{(for $l\neq 1$)},\
\end{cases}\\
&
\tilde{J}_1(k,l)=\frac{8l[(2k^2-1)E(k)+k'^2K(k)]}{3(2k^2-1)^{3/2}},\\
&
\tilde{J}_2(k,n,l)=
\begin{cases}
2\sqrt{2}\pi \nu\sech\left(\displaystyle\frac{n\pi K(k')}{2K(k)}\right)
 & \mbox{(for $l=1$ and $n$ odd)}; \\
0 & \mbox{(for $l\neq 1$ or $n$ even).}
\end{cases}
\end{align*}
Here $E(k)$ is the complete elliptic integral of the second kind 
 and $k'=\sqrt{1-k^2}$ is the complimentary elliptic modulus.
When $\delta\neq 0$, 
 the subharmonic Melnikov functions $M_\pm^{n/l}(\tau)$ and $\tilde{M}^{n/l}(\tau)$
 are not identically zero for any relatively prime integers $n,l>0$
 since $J_1(k,l)$ and $\tilde{J}_1(k,l)$ are not zero.
Moreover, the homoclinic Melnikov function \eqref{eqn:homM} for $x_\pm^\h(t)$ is
\[
M_\pm(\tau)
 =-\frac{4}{3}\delta\pm\sqrt{2}\pi\nu\beta\csch\left(\frac{\pi\nu}{2}\right)\sin\tau,
\]
which is not identically zero for $\beta\neq0$.
See \cite{GH83,W90} for the computations of the Melnikov functions.

Let
\begin{align*}
&
R=\{k\in(0,1)\mid \mbox{$k$ satisfies \eqref{eqn:resk} for $n,l\in\Nset$}\},\\
&
\tilde{R}=\bigl\{k\in\bigl(1/\sqrt{2},1\bigr)\mid
\mbox{$k$ satisfies \eqref{eqn:tresk} for $n,l\in\Nset$}\bigr\},
\end{align*}
and let
\begin{align*}
&S_\pm^k=\{(x_\pm^k(t),\theta)\in\Rset^2\times\Sset^1\mid t\in[0,\hat{T}^k),\theta\in\Sset^1\},\\
&
\tilde{S}^k=\{(\tilde{x}^k(t),\theta)\in\Rset^2\times\Sset^1\mid t\in[0,\tilde{T}^k),\theta\in\Sset^1\},\\
&
\Gamma_\pm=\{x_\pm^\h(t)\in\Rset^2\mid t\in\Rset\}\cup\{0\}.
\end{align*}
Noting that
\[
\lim_{n\to\infty}\tilde{M}^{2n+1/1}(\tau)=M_+(\tau)+M_-(\tau)
\]
and applying Theorems~\ref{thm:sub-mel1}, \ref{thm:hom-mel1}, \ref{thm:hom-mel2} and their slight extensions,
 we have the following.

\begin{prop}
\label{prop:Duf1}
The system~\eqref{DuffingAut} with $a=1$ has no real-analytic first integral
depending analytically on $\epsilon$ in neighborhoods of $S_\pm^k$ for $k\in R$,
 of $\tilde{S}^k$ for $k\in\tilde{R}$, and of $S_\pm^\h$ near $\epsilon=0$
if $\delta\neq0$.
\end{prop}

\begin{prop}
\label{prop:Duf2}
Let $\hat{U}_\pm$ (resp. $\tilde{U})$ be regions (resp. a region) in $\Rset^2$
 such that $\partial\hat{U}_\pm\supset\Gamma_\pm$ (resp. $\partial\tilde{U}\supset\Gamma_+\cup\Gamma_-$)
 and $\hat{U}_\pm\supset\{x_\pm^k(t)\mid t\in[0,\hat{T}^\alpha) \}$
 (resp. $\tilde{U}\supset\{\tilde{x}^k(t)\mid t\in[0,\tilde{T}^k)\}$
for some $k\in(0,1)$ (resp. $k\in(1/\sqrt{2},1)$).
For $|\epsilon|\neq 0$ sufficiently small
 the system~\eqref{DuffingAut} with $a=1$ is not real-analytically integrable
in the regions $\hat{U}_\pm\times\Sset^1$ (resp. $\tilde{U}\times\Sset^1$) 
 in the meaning of Theorem~\ref{thm:main-integ} if $\beta\neq0$.
\end{prop}

If $\beta\neq0$ and
\begin{align}\label{eqn:ineq}
\frac{\delta}{\beta}<\frac{3}{4}\sqrt{2}\pi\nu\csch \left(\frac{\pi\nu}{2}\right),
\end{align}
then $M_\pm(\tau)$ has a simple zero,
 so that for $|\epsilon|>0$ sufficiently small there exist transverse homoclinic orbits
 to a periodic orbit near the origin and chaotic dynamics may occur
 in \eqref{DuffingAut} with $a=1$, as stated in Section~\ref{Melnikov}.
From Proposition~\ref{prop:Duf2} we see that the system \eqref{DuffingAut} is nonintegrable
 in the meaning of Theorem~\ref{thm:main-integ}
 even if condition \eqref{eqn:ineq} does not hold, i.e., 
 there may exist no transverse homoclinic orbit to the periodic orbit, as stated in Remark~\ref{rmk:5b}(i).
On the other hand,  when the system \eqref{Duffing} is Hamiltonian, i.e., $\delta=0$,
 condition \eqref{eqn:ineq} always holds and such inconsistency does not occur.
See also Remark~\ref{rmk:5b}(ii).


We turn to the case of $a=-1$.
When $\epsilon=0$, in the phase plane
 there exists a one-parameter family of periodic orbits
\begin{align*}
\hat{x}^k(t)
 =&\biggl(\frac{\sqrt{2}k}{\sqrt{1-2k^2}}\cn\left(\frac{t}{\sqrt{1-2k^2}}\right),\\
& \quad-\frac{\sqrt{2}k}{1-2k^2}\sn\left(\frac{t}{\sqrt{1-2k^2}}\right)
\dn\left(\frac{t}{\sqrt{1-2k^2}}\right)\biggr),\quad
k\in\bigl(0,1/\sqrt{2}\bigr),
\end{align*}
as shown in Fig.~\ref{fig:Duffing}(b),
 and their period is given by $\hat{T}^k=4K(k)\sqrt{1-2k^2}$.
See \cite{Y94,Y96}.
Assume that the resonance conditions
\begin{equation}
l\hat{T}^k=\frac{2\pi n}{\nu},\quad\mbox{i.e.,}\quad
\nu=\frac{\pi n}{2lK(k)\sqrt{1-2k^2}}
\label{eqn:hresk}
\end{equation}
holds for $l,n>0$ relatively prime integers.
We compute the subharmonic Melnikov function \eqref{eqn:subM}
 for $\hat{x}^k(t)$ as
\[
\hat{M}^{n/l}(\tau)=-\delta\hat{J}_1(k,l)\pm\beta\hat{J}_2(k,n,l)\sin\tau,
\]
where
\begin{align*}
&
\hat{J}_1(k,l)=\frac{8l[(2k^2-1)E(k)+k'^2K(k)]}{3(1-2k^2)^{3/2}},\\
&
\hat{J}_2(k,n,l)=
\begin{cases}
\displaystyle
\frac{\sqrt{2}\pi^2n}{K(k)\sqrt{1-2k^2}}
\sech\left(\frac{\pi nK(k')}{2K(k)}\right) & \mbox{(for $l=1$ and $n$ odd)};\\
0\quad & \mbox{(for $l\neq 1$ or $n$ even)}.
\end{cases}
\end{align*}
See also \cite{Y94,Y96} for the computations of the Melnikov function.
When $\delta\neq 0$, the subharmonic Melnikov function $\hat{M}^{n/l}(\tau)$ is not identically zero
 for any relatively prime integers $n,l>0$ since $J_1(k,l)$ is not zero.

Let
\begin{align*}
&
\hat{R}=\bigl\{k\in\bigl(0,1/\sqrt{2}\bigr)\mid
 \mbox{$k$ satisfies \eqref{eqn:hresk} for $n,l\in\Nset$}\bigr\},
\end{align*}
and let
\[
\hat{S}^k=\{(\hat{x}^k(t),\theta)\in\Rset^2\times\Sset^1\mid t\in[0,\hat{T}^k),\theta\in\Sset^1\}.
\]
Applying Theorem~\ref{thm:sub-mel1}, we obtain the following.

\begin{prop}\label{prop:5b}
The system~\eqref{DuffingAut} with $a=-1$
 has no real-analytic first integral depending analytically on $\epsilon$
 in a neighborhood of $\hat{S}^k$ for $k\in\hat{R}$ 
 near $\epsilon=0$
 if $\delta\neq 0$
.
\end{prop}

\begin{rmk}\
\begin{enumerate}
\setlength{\leftskip}{-2em}
\item[(i)]
Since the subharmonic Melnikov function $\hat{M}^{n/l}(\tau)$ is constant for $l\neq1$,
 Theorem~\ref{thm:sub-mel2} is not applicable to \eqref{DuffingAut} with $a=-1$.
So we cannot exclude the possibility that
 the system \eqref{DuffingAut} with $a=-1$ is $(3,0)$-integrable
 in the meaning of Theorem~\ref{thm:main-integ} when $\beta,\delta\neq 0$.
\item[(ii)]
It was shown in \cite{Y21a} that the system~\eqref{DuffingAut} with $a=-1$ is
meromorphically nonintegrable in a meaning similar to that of Theorem~\ref{thm:main-integ} when the independent and state variables are extended to complex ones.
\end{enumerate}
\end{rmk}

\section*{Acknowledgement}
This work was partially supported by the JSPS KAKENHI Grant Numbers JP17H02859 and JP19J22791.

 

\appendix
\renewcommand{\theequation}{\Alph{section}.\arabic{equation}}
\setcounter{equation}{0}

\section{Proof of Proposition~\ref{fcn-ind}}
In this Appendix, we prove Proposition~\ref{fcn-ind}.
We begin with the following lemma.

\begin{lem}\label{lem:rank}
Let $\Omega$ be an open subset of $\Rset^k$
 and let $\chi_j:\Omega\to \Rset$, $j=1,\ldots,m$, be analytic, where $k,m\in\Nset$.
Let $\chi(x)=(\chi_1(x),\ldots,\chi_m(x))$.
If $\rank d\chi$ is constant on $\Omega$ and less than $m$,
 then for any $x\in \Omega$ there exists a neighborhood $V$ of $x$
 on which $\chi_1, \ldots,\chi_m$ are \emph{analytically dependent},
 i.e., there exist an open set $\Omega' \subset \Rset^m$
 and a non-constant analytic map $\zeta:\Omega'\to\Rset$
 such that $\chi(V)\subset \Omega'$ and $\zeta(\chi(y))=0$ for any $y\in V$.
\end{lem}

\begin{proof}
Using Theorem~1.3.14 
 of \cite{N68}
 and an argument in the proof of Theorem~1.4.15 of \cite{N68},
 we can immediately obtain the desired result as follows.
The theorem says that
 there exist a neighborhood $V$ (resp. $V'$) of $x$ (resp. of $\chi(x)$),
 a cube $Q$ (resp. $Q'$) in $\Rset^k$ (resp. in $\Rset^m$)
 and analytic isomorphisms $u:Q\to V$ and $u':V'\to Q'$
 such that the composite map $u'\circ\chi\circ u$
 has the form $(x_1, \ldots, x_k)\to (x_1, \ldots, x_{m'}, 0, \ldots, 0)$,
 where $x_j$ is the $j$th element of $x$ for $j=1,\ldots,k$ and $m'=\rank d\chi<m$.
Here a \emph{cube} in $\Rset^k$ is an open set  of the form
\[
\{x\mid |x_j-a_j|<r_j,j=1,\ldots,k\}
\]
for some $a_j\in\Rset$ and $r_j>0$, $j=1,\ldots,k$.
Letting $u'=(u'_1, \ldots, u'_m)$ and $\zeta=u'_m$,
 we have $\zeta(\chi(y))=0$ for every $y\in V$.
\end{proof}

Let $f_\epsilon: \M\to\Rset$ be an analytic function
 such that it depends on $\epsilon$ analytically.
We expand it near $\epsilon=0$ as $f_\epsilon(x)=\sum_{j=0}^\infty f^j(x)\epsilon^j$,
 where $f^j(x)$, $j\in\Zset_0:=\Nset\cup\{0\}$, are analytic functions on $\mathscr{M}$.
Define the order function $\sigma(f_\epsilon)$ by
\[
\sigma(f_\epsilon):=\min\{j\in\Zset_0
 \mid f^j(x)\not\equiv 0\}
\]
if $f_\epsilon\not\equiv 0$
 and $\sigma(0):=+\infty$, as in \cite{BCRS96}.

\begin{lem}\label{fcn-ind:step1}
Suppose that
 $f_\epsilon(x)$ is a nonconstant analytic first integral of \eqref{eqn:p-gsys}
  depending analytically on $\epsilon$ near $\epsilon=0$.
Then there exists an analytic first integral
 $\tilde{f}_\epsilon(x)=\tilde{f}^0(x)+O(\epsilon)$ depending analytically on $\epsilon$ near $\epsilon=0$
 such that $\tilde{f}^0(x)$ is not constant.
\end{lem}

\begin{proof}
Since $f_\epsilon$ is not constant, $\sigma(df_\epsilon)$ takes a finite value.
Let $k=\sigma(df_\epsilon)$ and $f_\epsilon(x)=\sum_{j=0}^\infty \epsilon^j f^j(x)$.
Define
\begin{align*}
\tilde{f}_\epsilon(x):=\frac{1}{\epsilon^k}\left(f_\epsilon(x)-\sum_{j=0}^{k-1}\epsilon^j f^j(x)\right).
\end{align*}
Then $\tilde{f}_\epsilon(x)=f^k(x)+O(\epsilon)$ and $\tilde{f}^0(x)=f^k(x)$ is not constant.
Moreover, $\tilde{f}_\epsilon(x)$ is a first integral of \eqref{eqn:p-gsys}
 since $X_\epsilon(f_\epsilon)=0$ and $\sum_{j=0}^{k-1}\epsilon^j f^j$ is constant.
\end{proof}

\begin{proof}[Proof of Proposition~\ref{fcn-ind}]
Modifying the proof of Ziglin's lemma \cite{Ad,CR88,Z82a} slightly, we prove this proposition.
For $k=1$
 the statement of the proposition holds by Lemma~\ref{fcn-ind:step1}.
Let ${k}>1$ and suppose that it is true up to $k-1$. 
Let $G_1^\epsilon(x),\ldots,G_k^\epsilon(x)$
 be analytic first integrals of \eqref{eqn:p-gsys} in a neighborhood of $F^{-1}(c)$ near $\epsilon=0$
 such that they are functionally independent for $\epsilon\neq 0$
 and depend analytically on $\epsilon$.
Without loss of generality,
 we assume that $G_1^0(x),\ldots, G_{k-1}^0(x)$ are functionally independent near $F^{-1}(c)$.
Letting $G^\epsilon(x)=(G_1^\epsilon(x),\ldots,G_k^\epsilon(x))$, we see that $G^0(\varphi(I,\theta))$ depends only on $I$, as in the proof of Lemma~\ref{erg}(i)\green{,}
 where $\varphi$ denotes the analytic diffeomorphism
  in Proposition~\ref{LAJ}(ii). Let $\tilde{G}_j(I)=G_j^0(\varphi(I,\theta))$ for $j=1,\ldots,k$.
Note that, if $d\tilde{G}_1(I), \ldots, d\tilde{G}_k(I)$ are linearly independent at $I=I_0\in U$,
then so are $dG_1(x), \ldots, dG_k(x)$ on $\varphi(\{I_0\}\times \Tset^q)\subset \mathcal{U}$.

Assume that $\tilde{G}_1(I),\ldots,\tilde{G}_{k-1}(I), \tilde{G}_k(I)$ are functionally dependent
 in an open set $U'\subset U$.
So $\Omega:=\{p\in U'\mid\rank d_p\tilde{G}=k-1\}$
 contains a dense open set in $U'$
 since 
 $d\tilde{G}_1(I), \ldots, d\tilde{G}_{k-1}(I)$ are functionally independent on $U$.
By Lemma~\ref{lem:rank}, there exist an open set $\Omega'\subset\Rset^{k}$
 and a nonzero analytic function $\zeta:\Omega'\to\Rset$
 such that $\tilde{G}(V)=(\tilde{G}_1(V),\ldots, \tilde{G}_k(V))\subset\Omega'$ and
\begin{equation*}
\zeta(\tilde{G}_1(I),\ldots,\tilde{G}_k(I))=0
\end{equation*}
in a neighborhood $V$ of $p\in\Omega$.
Moreover, there is a positive integer $s$ such that $({\partial^s \zeta}/{\partial y_k^s})(\tilde{G}(I))\neq0$,
since if not,
then $\zeta(\tilde{G}_1(I),\ldots,\tilde{G}_{k-1}(I),y_k)$ depends on $y_k$ near $\tilde{G}(V)$
and consequently $\tilde{G}_1(I),\ldots,\tilde{G}_{k-1}(I)$ are functionally dependent.
Let $s$ be the smallest one of such integers and let $\tilde{\zeta}(y)=({\partial^{s-1} \zeta}/{\partial y_k^{s-1}})(y)$. Then $\tilde{\zeta}$ satisfies
\begin{equation*}
\tilde{\zeta}(\tilde{G}_1(I),\ldots,\tilde{G}_k(I))=0
\end{equation*}
and $(\partial\tilde{\zeta}/\partial y_k)(\tilde{G}(I))\neq 0$ on $V$.
Hence,
\begin{equation}\label{xi3}
\tilde{\zeta}(G^0(x))=0
\end{equation}
and $(\partial\tilde{\zeta}/\partial y_k)(G^0(x))\neq 0$ on $\varphi(V\times \Tset^q)$.

Let $\hat{G}_k^\epsilon(x)=\tilde{\zeta}(G^\epsilon(x))/\epsilon$.
By \eqref{xi3}
 $\hat{G}_k^\epsilon$ is an analytic first integral depending analytically on $\epsilon$.
We have
\begin{align*}
d\hat{G}_k^\epsilon=\epsilon^{-1}d(\tilde{\zeta}(G^\epsilon(x)))
=\epsilon^{-1}\sum_{j=1}^k \frac{\partial \tilde{\zeta}}{\partial y_j}(G^\epsilon(x)) dG_j^\epsilon
\end{align*}
and
\begin{align*}
N(\hat{G}^\epsilon)
:=dG_1^\epsilon\wedge\ldots\wedge dG_{k-1}^\epsilon\wedge d\hat{G}_k^\epsilon
=\epsilon^{-1} \frac{\partial\tilde{\zeta}}{\partial y_k}(G^\epsilon(x))
 dG_1^\epsilon\wedge\ldots\wedge d{G}_k^\epsilon.
\end{align*}
Since $(\partial\tilde{\zeta}/\partial y_k)(G^0(x))\neq 0$ on $\varphi(V\times \Tset^q)$,
 we have $\sigma((\partial \tilde{\zeta}/\partial y_k)(G^0(x)))=0$, so that
\begin{align*}
\sigma(N(\hat{G}^\epsilon))=\sigma(\epsilon^{-1}N({G}^\epsilon))
+\sigma\left(\frac{\partial \tilde{\zeta}}{\partial y_k}(G^\epsilon(x))\right)
=\sigma(N({G}^\epsilon))-1.
\end{align*}
Repeating this procedure
 till $\sigma(N(\hat{G}^\epsilon))=0$, we obtain
\[
dG_1^0\wedge\ldots\wedge dG_{k-1}^0\wedge d\hat{G}_k^0\neq0,
\]
which means the desired result.
\end{proof}


\end{document}